\documentclass[12pt]{amsart}
\usepackage{amssymb, amsmath, adjustbox, enumerate, amsbsy, stmaryrd, stackrel, nccmath}
\usepackage{amsmath}
\usepackage{mathtools}
\usepackage{amsfonts}
\usepackage[numbers]{natbib}
\usepackage{mathtools}

\usepackage[pagewise]{lineno}

\usepackage[utf8]{inputenc}
\newcommand\reallywidehat[1]{%
\savestack{\tmpbox}{\stretchto{%
  \scaleto{%
    \scalerel*[\widthof{\ensuremath{#1}}]{\kern.1pt\mathchar"0362\kern.1pt}%
    {\rule{0ex}{\textheight}}
  }{\textheight}%
}{2.4ex}}%
\stackon[-6.9pt]{#1}{\tmpbox}%
}
\usepackage{amsmath,wasysym}
\usepackage{comment}
\usepackage{amsmath}
\usepackage[]{enumitem} 
\usepackage{caption}
\usepackage[mathscr]{eucal}
\usepackage{amsthm}
\usepackage{graphicx}
\usepackage{amscd}
\usepackage{color}
\usepackage{hyperref}

    \hypersetup{
    colorlinks=true, 
    linktoc=all,     
    linkcolor=blue,  
    citecolor=blue,
}

\makeatletter
\newcommand{\mainsectionstyle}{%
  \renewcommand{\@secnumfont}{\bfseries}
  \renewcommand\section{\@startsection{section}{2}%
    \z@{.5\linespacing\@plus.7\linespacing}{-.5em}%
    {\normalfont\bfseries}}%
}
\makeatother

\usepackage{verbatim} 

\newtheorem{prop}{Proposition}[section]

\newtheorem{coro}[prop]{Corollary}
\newtheorem{defi}[prop]{Definition}

\newtheorem{lemm}[prop]{Lemma}
\newtheorem{nota}[prop]{Notation}
\newtheorem{pf-thm}[prop]{proof of theorem}
\newtheorem{rema}[prop]{Remark}

\newtheorem{theo}[prop]{Theorem}

\newtheorem*{ack}{Acknowledgments}

\def\rank{\mathrm{rank}}

\def\tr{\mathrm{tr}}

\def\trp{\mathrm{tr}_g^\perp}

\DeclarePairedDelimiterX{\norm}[1]{\lVert}{\rVert}{#1}
 
\newcommand{\Hquad}{\hspace{0.5em}} 

\makeindex

\numberwithin{equation}{section}

\title{CONFORMAL EMBEDDINGS VIA HEAT KERNEL}

\date{}

\author{Zhitong Su}

\newcommand{\Addresses}{{
  \bigskip
  \footnotesize
  \begin{flushright}Zhitong Su\\ 
  \textsc{
  MOE-LCSM\\ School of Mathematics and Statistics\\ Hunan Normal University\\
  Changsha 410081, P. R. China}\par
  \textit{E-mail address}:  \texttt{suzht@hunnu.edu.cn}\end{flushright}

}}

\allowdisplaybreaks

\begin{document}

%

\maketitle

\begin{abstract}
For any n-dimensional compact Riemannian Manifold $M$ with smooth metric $g$, by employing the heat kernel embedding introduced by B\'erard-Besson-Gallot (1994, \cite{BBG}), we intrinsically construct a canonical $t$-family of conformal embeddings $C_{t,k}$: $M\rightarrow\mathbb{R}^{q(t)}$, with $t>0$ sufficiently small, $q(t)\gg t^{-\frac{n}{2}}$, and $k$ as a function of $O(t^l)$ with $l\geq 2$ in proper sense. Our approach involves finding all these canonical conformal embeddings, which shows the distinctions from  the  isometric embeddings introduced by Wang-Zhu (2015, \cite{WZ}). 
\end{abstract}

\section{Introduction}
Let $(M,g)$ be an n-dimensional compact Riemannian manifold, the following classical problem, called  the \textbf{isometric embedding problem} is studied in differential geometry. Does there exist an embedding $u:M\longrightarrow \mathbb{R}^N$ for some $N$ such that 
\begin{equation*}
    u^*g_{\mathrm{can}}=g,
\end{equation*}
where $g_{\mathrm{can}}$ is the Euclidean metric in $\mathbb{R}^N$?
In 1956, J. Nash famously proved in \cite{N2} that there exists a $C^r$-class isometric embedding for $g\in C^r$, with $r\geq 3$ or $r=\infty$. Furthermore, for any compact $n$-dimensional Riemannian manifold, the optimal value of $N$ he found was $N=\frac{3}{2}n(n+1)+4n$.\par

In \cite{N2}, Nash developed an iteration nowadays known as the Nash-Moser theorem to address the problem of losing differentiability when taking the usual Newton iteration.
Decades later, M.G\"unther (1989, \cite{G1}) significantly simplified Nash's proof by applying an elliptic operator to develop a different iteration, which avoids the loss of differentiability (see Section \ref{Gunther}). This allows one to simply use the usual Banach fixed point theorem to conclude the proof. His approach is also exposed in the proceedings \cite{G2} of  ICM 1990 Kyoto. \par

Nash and G\"unther's construction of the isometric embedding is highly flexible. By employing this method, any $C^{r\geq 3}$ embedding $u: M \longrightarrow \mathbb{R}^N$ such that the induced metric is less than or equal to $g$ can serve as a start to produce an isometric embedding. This great flexibility, on the other hand, often results in the isometric embeddings being \textbf{noncanonical}.\par
Contrastingly, in 1994, B\'erard, Besson, and Gallot \cite{BBG} constructed an  `asymptotically isometric' embedding using  the heat kernel of the manifold. This embedding, referred to as the \textbf{(normalized) heat kernel embedding} throughout, maps a compact Riemannian manifold $M$ into $\ell^2$, the space of square summable series, and is  constructed as follows:
\begin{equation*}
    \Psi_t: x\mapsto \sqrt{2}(4\pi )^{\frac{n}{4}} t^{\frac{n+2}{4}}\cdot \big\{ e^{-\lambda_j t/2}\phi_j(x) \big\}_{j\geq 1}, \text{ for } t>0,
\end{equation*}
where $\lambda_j$ is the $j$th eigenvalue of the Laplacian $\Delta=\mathrm{tr}_g \nabla^2$ of $(M,g)$, here $\nabla$ is the Levi-Civita connection, and $\{\phi_j\}_{j\geq 0}$ is an $L^2$ orthonormal eigenbasis of $\Delta$. It is worth noting that the embedding $\Psi_t$ is \textbf{canonical} due to the fact that it is constructed by the heat kernel and therefore the spectral geometry of $(M,g)$ uniquely determines it. A more precise formula in \cite{BBG} that justifies the above statement is the following, indicating that $\Psi_t$ tends to an isometry in the following sense:
\begin{equation*}
    \Psi^*_t g_{\mathrm{can}}= g+\frac{t}{3} (\frac{1}{2}\mathrm{Scal}_g\cdot g-\mathrm{Ric}_g)+O(t^2),
\end{equation*}
where the $g_\mathrm{can}$ is the standard Euclidean metric in $\ell^2$, $\mathrm{Scal}_g$ is the scalar curvature of $(M,g)$, $\mathrm{Ric}_g$ is the Ricci curvature of $(M,g)$, and the convergence is in the $C^r$ sense for any $r>0$.\par
In light of the facts that Nash and G\"unther's isometric methods (\cite{N2},\cite{G1}) being flexible but far from being canonical, and  B\'erard, Besson, and Gallot's heat kernel embedding (\cite{BBG}) being canonical but not yet exactly isometric,  Wang and Zhu (2015, \cite{WZ}) embarked on a study aimed at finding a canonical isometric embedding of a compact Riemannian manifold into $\mathbb{R}^q$ for $q\gg 1$ by using the heat kernel of this manifold.  Their approach begins by first modifying the  heat kernel embedding $\Psi_t$ in \cite{BBG} to a better approximation with an error term of $O(t^l)$ for any $l\geq 2$, and  continues by perturbing such an `almost isometric' embedding to an isometric one. Namely for any $l\geq 1$, by using the $\Psi_t$, they find a canonical family of `almost' isometric embeddings $\Tilde{\Psi}_t: M \longrightarrow \ell^2$ such that
\begin{equation*}
    \Tilde{\Psi}_t^* g_{\mathrm{can}}=g+O(t^l)
\end{equation*}
in the $C^r$ sense for $r>0$. Subsequently, they find a unique $C^{r,\alpha}$ isometric embedding $I_t:M\longrightarrow\mathbb{R}^{q(t)}$ such that 
\begin{equation*}
    \norm{ I_t-\Tilde{\Psi}_t}_{C^{r,\alpha}(M)}= O(t^{l+\frac{1}{2}-\frac{r+\alpha}{2}}),
\end{equation*}
where $q(t)\geq t^{-\frac{n}{2}-\rho}$, $\rho>0$, $r+\alpha<l+\frac{1}{2}$, $0<\alpha<1$, $r\geq 2$, $t\in (0,t_0)$ for some $t_0>0$ depending on $r,\alpha, l,$ and $g$. Additionally, we note  B\'erard, Besson, and Gallot's  heat kernel embeddings can be used in many other ways. For further references, historical contexts, and other uses of this heat kernel embedding, see \cite{Pt} and \cite{Te}.\par

From the  view of K\"ahler geometry and complex geometry,  one may seek more embeddings of this canonical type. In analogy to the Kodaira embedding (see, e.g. \cite{GH}) in K\"ahler geometry that preserves the holomorphic structure, in the current paper we find a family of canonical embeddings of compact Riemannian manifolds that preserve the \textbf{conformal} structure. Indeed, an isometric one is already a conformal one, but starting with the heat kernel `almost' isometric embedding and looking into G\"unther's method, we have shown that by requesting the result map to be conformal and keeping each step done canonically, one can find  \textbf{a  family of canonical conformal embeddings} of $(M,g)$ in Euclidean space, with the isometric embedding constructed in \cite{WZ} as one special case among them. Note that the present paper partially fulfills the proposal of Wang and Zhu, specifically the part concerning the construction of canonical conformal embeddings, see \cite[Introduction]{WZ}.

\par
Throughout, conformal embeddings are referred to embeddings that are conformal maps, see Definition \ref{conformal embeddings}.
In the following, we present the main theorems,  fixing the constant $\rho>0$, and $0<\alpha<1$, and using Einstein summation notation throughout.

\begin{theo}[Proposition \ref{part 1 of mt}] \label{main thm 1.1}
Let $(M,g)$ be a smooth n-dimensional compact Riemannian manifold without boundary, $g$ be the smooth Riemannian metric of $M$. Then for any integer $l\geq 2$ and a finite sequence of functions $\eta_1, \cdots, \eta_{l-1} \in C^\infty(M,\mathbb{R})$, there exists a $t$-family of \emph{canonical  almost conformal} embeddings $\Psi_{t,g(t),\eta_i}:M\rightarrow \ell^2$, such that
    \begin{equation*}
        (\Psi_{t,g(t),\eta_i})^*g_{\mathrm{can}}-\frac{\mathrm{tr}_g(\Psi_{t,g(t),\eta_i})^*g_{\mathrm{can}}}{n}g=O(t^l)
    \end{equation*}
    as $t\rightarrow 0_+$, where the convergence is in $C^r(M,\mathbb{R})$ sense for any $r\geq 0$.
\end{theo}
Note here that $\Psi_{t,g(t),\eta_i}$ is an \textit{almost conformal} embedding since the error term $O(t^l)$ is small when $t\rightarrow0_+$. And it is a \textit{canonical} embedding in the sense that it is determined by the geometry of $(M,g)$.\par
As we will see in Proposition \ref{part 1 of mt}, for given $\eta_i\in C^\infty (M,\mathbb{R})$ and each $1\leq i\leq l-1$,  we will uniquely determine an  $h_i\in \Gamma(\mathrm{Sym}^{\otimes 2}(T^*M))$ depending only on $(M,g)$ and $\eta_i$, such that \eqref{system of solving h_i} holds, and especially $\frac{\tr_g h_i}{n}=\eta_i$. Then for the metric $g(t):=g+\sum\limits_{i=1}^{l-1}h_it^i$, the map  $\Psi_{t,g(t),\eta_i}$ is defined as the heat kernel embedding (in the sense of \cite{BBG}) of $(M,g(t))$.\par
Given this $\Psi_{t,g(t),\eta_i}$, as quoted verbatim from \cite{WZ}, we have the following definition:
    \begin{defi}\emph{(Truncated embedding)}\label{truncated embedding}
Let 
\begin{equation*}
    \Pi_q :\ell^2 \longrightarrow \mathbb{R}^q
\end{equation*}be the projection of $\ell^2$ to the first $q$ components. To get a finite-dimensional almost conformal embedding, we introduce the \textbf{\emph{truncated embedding}}
\begin{equation*}
    \Psi^{q(t)}_{t,\eta_i} :=\Pi_q\circ\Psi_{t,g(t),\eta_i}: (M,g)\longrightarrow\ell^2\overset{\Pi_{q(t)}}{\longrightarrow} \mathbb{R}^{q(t)}.
\end{equation*}

\end{defi}
The following is the second part of our main theorem.
\begin{theo}[Proposition \ref{conformal immersion} and \ref{Inject}]\label{mainthm 1.3}
    
    Under the proceeding assumption, we have:
    \item For any integer $r\geq 2$ and $l$ satisfying $r+\alpha <l+\frac{1}{2}$, there exists a constant $t_0>0$ depending on $r,\alpha, l, g$ and $\eta_i $, such that for any $0<t\leq t_0$, there exists a family  of conformal embeddings $C_{t,k_t}$, parametrized by $k_t\in K:=\{ k_t\in C^{r,\alpha}(M,\mathbb{R})| \| k_t \|_{C^{r,\alpha}(M)}= O(t^l) \}$,   such that for  any  $k_t\in K$,  each truncated embedding $\Psi_{t,\eta_i}^{q(t)}$  can be perturbed to a unique $C^{r,\alpha}$ conformal embedding 
    \begin{equation*}
        C_{t,k_t} : M\rightarrow\mathbb{R}^{q(t)},
    \end{equation*}
    where the dimension $q(t)\geq t^{-\frac{n}{2}-1}$.
    
    Moreover, the resulting conformal map satisfies the estimate:
   \begin{gather*}
       \|C_{t,k_t}-\Psi_{t,g(t),\eta_i}\|_{C^{r,\alpha}} = O(t^{l+\frac{1-r-\alpha}{2}}),\\
            \|C_{t,k^{a}_t}-C_{t,k^{b}_t} \|_{C^{r,\alpha}}< C(r,\alpha,M,g,\eta_i)t^{-\frac{r+\alpha}{2}} \|k_t^a-k_t^b\|_{C^{r,\alpha}}, \quad \forall\, k_t^a,k_t^b \in K.
    \end{gather*} 

\end{theo}

\begin{rema}[on $t^l$]
    The requirement $l+\frac{1}{2}>r+\alpha$ is crucial for Theorem \ref{mainthm 1.3}, which can be viewed as an implicit function theorem following Wang-Zhu \cite{WZ}.  As will be seen in \eqref{importance on l}, the final step of applying the Banach fixed point theorem fails without $l+\frac{1}{2}>r+\alpha$. This is why we first prepare Theorem \ref{main thm 1.1} that  perturbs the metric $g$ to ensure the pullback metric $(\Psi_{t,g(t),\eta_i})^*g_{\rm can}$ has an error of order $t^l$; this technique originates from  \cite{WZ} as well. 
\end{rema}

\begin{rema}

In this context, we start with the `almost' isometric heat kernel embedding $\Psi_t$ in \cite{BBG}, to construct conformal embeddings $C_{t,k}$, which encapsulates only the intrinsic information of $(M,g)$. 
We emphasize that this intrinsic dependence leads to the key property of canonicity in $C_{t,k}$, which is our goal here,  even though the construction resembles that of an isometric embedding.

To address a potential concern, we note that one can certainly compose  $C_{t,k}$ with any M\"obius transformation of $\overline{\mathbb{R}^N}$, the one-point compactification of $\mathbb{R}^N$, to achieve another conformal embedding. However, such  operations are not determined by the intrinsic property of $(M,g)$ and thus do not preserve canonicity.\par
Due to the same reason, it is also noteworthy that finding the optimal dimension $q(t)$ is not our goal here, as lower dimensions can result in less canonical embeddings.
\end{rema}

 \par
 The main techniques in this article can be described as a process of `recovering the trace', and can be outlined as follows. Notice that a map $u:M \longrightarrow \mathbb{R}^N$ being \textbf{free} (see Definition \ref{def of P and P_c}) is a strong condition, which will ensure the existence and uniqueness of the solution to the equation  (see Lemma \ref{Free mapping's linear algebra lemma}):
\begin{equation}\label{0.5}
P(u)\cdot v=\begin{bmatrix}
  \xi& f
\end{bmatrix}^T,
\end{equation}
 where the definition of $P(u)$ is in Definition \ref{def of P and P_c}. However, to address our conformal question, we need to look into the equation
\begin{equation}\label{0.6}
    P_c(u)\cdot v=\begin{bmatrix}
  \xi & f-\frac{\tr_g f}{n}g
\end{bmatrix}^T,
\end{equation}  where $P_c(u)$ is obtained from $P(u)$ by subtracting its own trace of the second derivative part (see Definition \ref{def of P and P_c}). The challenge lies in the fact that $P_c(u)$ is not of full rank. To overcome this difficulty, we point out that the solution of \eqref{0.5} is one special solution of \eqref{0.6}. And precisely describing the kernel of $P_c(u)$, which corresponds to the trace that to be recovered, allows us to obtain all the solutions in the following manner: `solutions of \eqref{0.6}'= `one special solution' +`kernel of $P_c(u)$'(see Remark \ref{essenttial remark}). In this context the `one special solution' corresponds to the isometric embeddings attained by Wang and Zhu in \cite{WZ}, hence we will closely follow their construction.  This coincides with the general perspective that isometric embeddings are special cases of conformal ones.

\begin{ack}\normalfont
The author greatly appreciates his advisor Xiaowei Wang for suggesting the
question and for the enlightening discussion and constant support. He also thanks Bin Guo, Jacob Sturm, Xi-Ping Zhu, Wei Yuan, and Xian-Tao Huang for helpful discussions, and  Changzheng Li for thoughtful guidance. He thanks
the anonymous referees for careful readings as well as   valuable feedback, which helped to clarify some technical points. This paper serves as a part of the author's Ph.D thesis.

\end{ack}

\section{Heat kernel embeddings  and modifications to almost conformal maps}
Let $(M,g)$ be an n-dimensional compact Riemannian manifold with smooth metric $g$. Denote the eigenvalues of the Laplacian of $(M,g)$ as $0=\lambda_0<\lambda_1<\lambda_2<\cdots$, and let $\{\phi_j\}_{j\geq 0}\subset C^\infty (M)$ be a corresponding $L^2$-orthonormal basis of the real eigenfunctions. In other words, this means $\Delta_g \phi_j=-\lambda_j \phi_j$, and $\int_M \phi_i\phi_j d\mathrm{vol}_g=\delta_i^j$, for $i,j\geq 0$. The heat kernel of $(M,g)$ is:
\begin{equation*}
    H(x,y,t)=\sum\limits_{j=1}^\infty e^{-\lambda_j t} \phi_j(x)\phi_j(y),
\end{equation*}
where $x,y\in M$, and $t>0$. Recall the definition in \cite{BBG} regarding almost isometric heat kernel embeddings into $\ell^2$:
\begin{defi}\label{defi of psi_t}
We call the family of maps
\begin{equation*}
    \Phi_t: \begin{matrix}M&\longrightarrow&\ell^2\\
    x & \longmapsto & \{e^{-\lambda_jt/2}\phi_j(x)\}_{j\geq 1}
    \end{matrix} \quad \text{ for } t>0
\end{equation*}
the \emph{\textbf{heat kernel embeddings}},   and call $\Psi_t=\sqrt{2}(4\pi)^{n/4}t^{\frac{n+2}{4}}\cdot \Phi_t$ the \emph{\textbf{normalized heat kernel embeddings}}.
\end{defi}
One main theorem in \cite{BBG} can be phrased as the following:
\begin{theo}\label{thm-mainthm of BBG}
For $t\rightarrow 0_+$, there is an expansion
\begin{equation}\label{mainthm of BBG}
    \Psi^*_t g_{\mathrm{can}}=g+\sum\limits_{i=1}^l t^i A_i(g)+O(t^{l+1}),
\end{equation}
in the $C^r$ sense for any $r\geq 0$, with
\begin{equation*}
    A_1=\frac{1}{3}(\frac{1}{2}S_g\cdot g-\mathrm{Ric}_g),
\end{equation*}
where the $g_{\mathrm{can}}$ is the metric of $\ell^2$, $S_g$ is the Scalar curvature, the $A_i$'s are universal polynomials of the covariant differentiations of the metric $g$ and its curvature tensors up to order $2i$.
\end{theo}
Similar to Proposition 5 of \cite{WZ}, we can perturb the $t$-dependent map   to a family 
of $t$-dependent maps, each of which   is in the form of an almost conformal map. The idea, compared to the isometric case, is to require it to be isometric to some conformal metric $\lambda^2 g$, instead of the metric $g$ itself.
Here is the definition of conformal map upon which we base our understanding:
\begin{defi}\label{conformal embeddings}
    
Assume $f$ is an embedding from $(M,g_M)$ to $(N,g_N)$, which both are Riemannian manifolds, and $M$ is of dimension $m$. $f$ is a \emph{\textbf{conformal map}} from $M$ to $f(M)$, iff
\begin{equation} \label{conformal}
    f^*g_N-\frac{\mathrm{tr}_{g_M}f^*g_N}{m}g_M=0.
\end{equation}
 Note this is equivalent to define a conformal map as the $f$ satisfying $f^*g_N=\lambda^2 g_M$ for some function $\lambda\in C^\infty(M,\mathbb{R})$. An embedding that is a conformal map is called a \emph{\textbf{conformal embedding}}.
\end{defi}
With a slight abuse of language, we also refer to a smooth map or a smooth immersion $f$ as a \textit{conformal map} as long as it satisfies the equation (\ref{conformal}). Much attention will be focused on seeking immersions that satisfy (\ref{conformal}), and ultimately showing
the immersions are embeddings. Therefore, such an abuse won't affect our results.

\begin{rema}\label{denotion of C}
For a 2-tensor $\alpha \in \Gamma(T^*M \otimes T^*M)$, the commonly encountered term in this paper is the traceless part of $\alpha$, given by $ \alpha-\frac{\mathrm{tr}_g\alpha}{n}g$.
Thus, it is useful to introduce the following notation for abbreviation
   $$\trp (\alpha):=\alpha-\frac{\mathrm{tr}_g\alpha}{n}g,$$
  which we refer to as the \emph{\textbf{conformal linear operator}} or \emph{\textbf{traceless linear operator}}. Notice that if $\alpha\in\Gamma(\mathrm{Sym}^{\otimes 2}(T^*M))$, then $\tr_g^\perp (\alpha)\in\Gamma(\mathrm{Sym}^{\otimes 2}(T^*M))$. The use of  perpendicular $\perp$ is justified, as it satisfies $\langle \alpha-\frac{\tr_g \alpha}{n}g,\, \frac{\tr_g \alpha}{n}g \rangle =0$, where we employ the inner product of 2-tensors induced by $g$. 
\end{rema}

\begin{prop}\label{part 1 of mt}
For any $l\geq 2$,  $\eta_i\in C^{\infty}(M,g)$, $1\leq i\leq l-1$, there are $h_i\in \Gamma(\mathrm{Sym}^{\otimes2}(T^*M))$ uniquely determined by $\eta_i$ satisfying $\frac{1}{n}\tr_g h_i=\eta_i$,  such that for the family of metrics 
\begin{equation*}
    g(s)=g+\sum\limits_{i=1}^{l-1}s^ih_i,
\end{equation*}
the induced metric from the heat kernel embeddings with $\Psi_{t,g(s)}: (M,g(s))\rightarrow \ell^2$ satisfies the estimate
\begin{equation*}
    ||\Psi^*_{t,g(t)}g_{\mathrm{can}}-\frac{\tr_g\Psi^*_{t,g(t)}g_{\mathrm{can}}}{n}g||_{C^r(M,g)}\leq C(g,l,r) t^l,
\end{equation*}
for any $r\geq 0$, where the constant $C(g,l,r)$ depends only on $l,r$ and the geometry of $(M,g)$.
\end{prop}
\begin{proof}
This proposition is, as we mentioned, a conformal version of Proposition 5 in \cite{WZ},  with more attention to the trace part. We start by assuming the family of metrics $g(s)$ can be expressed as: $$g(s)=g+\sum_{i=1}^{l-1}h_i s^i \text{ with } h_i\in \Gamma(\mathrm{Sym}^{\otimes 2}(T^*M)).$$ Our objective is to determine the proper $h_i$'s. Let $G(s,t):=\Psi_{t,g(s)}^*g_{\rm can}=g(s)+tA_1(g(s))+t^2A_2(g(s))+\cdots$, where the $A_i$'s are given as in Theorem \ref{thm-mainthm of BBG}. Then after letting $s=t\rightarrow 0$, and define $A_{i,j}(h_1,\cdots,h_j):=\frac{\partial^j}{\partial s^j}\biggr\rvert_{s=0}\frac{1}{j!}A_i(g(s))$, we have

\begin{align}
     \Psi_{t,g(t)}^*&g_{\mathrm{can}}-\dfrac{\mathrm{tr}_g\Psi_{t,g(t)}^*g_\mathrm{can}}{n}\cdot g
     =G(s,t)-\dfrac{\mathrm{tr}_gG(s,t)}{n}\cdot g\Big|_{s=t}\label{prop (2.4)}\\
     =&\trp(g)+t(\trp(h_1))+t^2(\trp(h_2))+\cdots\notag\\
     &+t(\trp(A_1(g))) + t^2(\trp(A_{1,1}(h_1))) +t^3(\trp(A_{1,2}(h_2))) +\cdots        \notag\\
     &+t^2(\trp(A_2(g))) + t^3(\trp(A_{2,1}(h_1))) +t^4(\trp(A_{2,2}(h_2))) +\cdots +O(t^l),\notag
\end{align}
where surely $\trp(g)=0$. Then we need to find proper $h_i$ such that for $1\leq \tilde{l}\leq l-1$, all the terms of $t^{\tilde{l}}$ in \eqref{prop (2.4)} vanish: 
\begin{equation}\label{system of solving h_i}
\begin{aligned}
\trp (h_1)&=-\trp(A_1(g)),\\
\trp(h_2)&=-\trp(A_2(g))-\trp(A_{1,1}(h_1)),\\
  \cdots \cdots&=\cdots \cdots.
\end{aligned}
\end{equation}
Here, the $\tilde{l}$-th equation depends on $h_1,\cdots, h_{\tilde{l}-1}$; thus, the $h_i$ will be found inductively.\par
First, we shall study the first equation of \eqref{system of solving h_i}:
\begin{equation}\label{first equation to find freedom}
   \tr_g^\perp(h_1)=-\tr_g^\perp(A_1(g)).
\end{equation}
In fact, explicit  expression of all the solutions of $h_1$ can be obtained. Given the geometric meaning of $\tr_g^\perp$ as taking the traceless part of a symmetric 2-tensor, we point out that at each point $x\in M$, the kernel $\mathrm{Ker}((\trp)_x)\subset \mathrm{Sym}^{\otimes 2}T^*_x M$ as a vector subspace  is of 1 dimension, which corresponds to the trace part of a 2-tensor, and such 1 dimension is generated by $g_x$. The method we use to  find the solution reflects the discussion we have had on recovering the trace in the Intrduction.

\par 
The above arguments show that locally $\mathrm{Ker}(\tr_g^\perp)$ is generated by $g$. Given that $h_1 = -A_1 (g)$ is one of the solutions, we can  express all  $h_1$ satisfying \eqref{first equation to find freedom} as follows:

\begin{equation}\label{h1}
    h_1=-A_1(g)+\frac{\tr_gA_1(g)}{n}g+\eta_1\cdot g.
\end{equation}
Here $\eta_1\in C^\infty(M,\mathbb{R})$ is  a globally smooth function. The expression of $h_1$ in (\ref{h1}) as a solution implies that  $\eta_1=\frac{\tr_gh_1}{n}.$\par
Next, for each $h_1$ in the form of \eqref{h1}, after fixing one $\eta_1\in C^\infty(M,g)$,  the equation   $h_2-\dfrac{\mathrm{tr}_g    h_2      }{n}\cdot g=-A_2(g)+\dfrac{\mathrm{tr}_g A_2(g)}{n}\cdot g-A_{1,1}(h_1)+\dfrac{\mathrm{tr}_g A_{1,1}(h_1)}{n}\cdot g$ can be solved for $h_2$. As in the $h_1$ case, all the $h_2$ have to be in the following form:
\begin{equation*}
    h_2=-A_2(g)-A_{1,1}(h_1)+\frac{\tr_g (A_2(g)+A_{1,1}(h_1))}{n}g+\eta_2\cdot g,\quad \eta_2\in C^\infty (M,\mathbb{R}),
\end{equation*}
where the $\eta_2$ satisfies    $\eta_2=\frac{\tr_gh_2}{n}.$

\par
Now we can have an explicit expression of $h_i$, $1\leq i\leq l-1$ inductively. Thus the way of $g(t)=g+\sum_{i=1}^{l-1}h_i t^i$ approaches to $g$ is determined but for the trace about $\eta_i$ to be given. Then, $\Psi_{t,g(t)}$ will satisfy:
\begin{equation}
    \begin{aligned}
    (\Psi_{t,g(t)})^*g_\mathrm{can}=&g+t(\frac{1}{n}\tr_g A_1(g)+\eta_1) g+t^2(\frac{1}{n}\tr_g (A_2(g)+A_{1,1}(h_1))+\eta_2) g\\
    &+\cdots+t^{l-1}(\frac{1}{n}\tr_g\sum_{i+j} A_{i,j}(h_1,\cdots,h_j)+\eta_{l-1})g+O(t^l)     
    \end{aligned}
\end{equation}
in the $C^r$ sense for any $r\geq 0$, thus concludes the proof.
\end{proof}

\begin{defi}\emph{(Canonical almost conformal embedding).}\label{defi-canonical almost conformal embedding} Given $\eta_i\in C^\infty(M,g)$, we call the $\Psi_{t,g(t),\eta_i}: M \rightarrow \ell^2$ constructed above the  \emph{ \textbf{(modified)  conformal heat kernel embedding}}.

\end{defi}
Next,  to obtain the embedding into $\mathbb{R}^q$, we truncate off the terms beyond the first $q$ ones, in the following sense.
 \begin{defi}\emph{(Truncated embedding)}\label{defi of truncation}
Let 
\begin{equation*}
    \Pi_q :\ell^2 \longrightarrow \mathbb{R}^q
\end{equation*}be the projection of $\ell^2$ to the first $q$ components. To get a finite-dimensional almost conformal embedding, we introduce the \textbf{\emph{truncated embedding}}
\begin{equation*}
    \Psi_{t,\eta_i}^{q(t)} :=\Pi_q\circ\Psi_{t,g(t),\eta_i}: (M,g)\longrightarrow\ell^2\overset{\Pi_{q(t)}}{\longrightarrow} \mathbb{R}^{q(t)}.
\end{equation*}

\end{defi}

The following Proposition estimates  the truncated tail approaches to $0$ exponentially, which is due to \cite[Proposition 9]{WZ}.
\begin{prop}[\protect{\cite[Proposition 9]{WZ}}]
\label{weyl law}
Consider a compact family $\{g_s\}_{s\in \Lambda}$ of smooth metrics defined on a compact $n$-dimensional Riemannian manifold $M$, where $g_s$ smoothly depends on the parameter $s$, and $\Lambda$ denotes this compact family of metrics. For any point $x$ in $M$, let $\{x^k\}_{1\leq k\leq n}$ represent the normal coordinates with respect to the metric $g_s$. Then, for any multi-indices $\Vec{\gamma}$ and $\vec{\beta}$, and any $q(t)$ satisfying $q(t)\geq t^{-(\frac{n}{2}+\rho)}$,

\begin{equation}\label{inequality of Weyl law}
\sum_{j\geq q(t)+1} e^{-\lambda_jt}D^{\Vec{\gamma}}\phi_jD^{\vec{\beta}}\phi_j \leq C e^{(-t^{-\frac{\rho}{n}})}.
\end{equation}
 The convergence is uniform across all points $x\in M$ and all metrics $s\in \Lambda$ in the $C^r$-norm for any $r\geq 0.$
\end{prop}

\begin{rema}
    Notice that Proposition \ref{weyl law} requires the family of metrics $\{g_s\}_{s\in \Lambda}$ to be compact. On the other hand, the family of metrics $g(t)$ in Proposition \ref{part 1 of mt} is given as 
    $$g(t)=g+\sum_i^{l-1} h_i t^i, \text{ with } \frac{\tr_g h_i}{n}=\eta_i,$$
    and each $h_i$ is determined by $h_1,\cdots,h_{i-1}$ inductively. To ensure that the  family of $g(t)$ is compact, and thus apply Proposition \ref{weyl law}, we need to either fix $\eta_i$ \emph{a priori} or to let $\eta_i$ vary in a compact family, and let $t$ in $[0, t_0]$ for some $t_0>0.$ 
\end{rema}

\begin{rema} As long as $\rho >0$, the right-hand side of \eqref{inequality of Weyl law} approaches to $0$ exponentially, and the inequality holds. For simplicity, we would take $\rho=1$ in the application of this paper, since finding the optimal $q(t)$ is not our goal.
\end{rema}
The following corollary applies the former discussion to our conformal case.
\begin{coro}\label{coro 2.10}
Given any $l\geq 2$, $\eta_i\in C^\infty (M,\mathbb{R})$  as in Proposition \ref{part 1 of mt}, for $q(t)\geq t^{-(\frac{n}{2}+1)}$, the truncated modified heat kernel embedding $\Psi_{t,\eta_i}^{q(t)}: (M,g) \rightarrow \mathbb{R}^{q(t)}$ still satisfies the asymptotic formula
\begin{equation}\label{error term Psi_^q(t)}
(\Psi_{t,\eta_i}^{q(t)})^*g_{\mathrm{can}} =\frac{\mathrm{tr}_g (\Psi_{t,\eta_i}^{q(t)})^*g_\mathrm{can}}{n}\cdot g +O(t^l)    
\end{equation}
in the $C^r$ sense for any $r\geq 0$.
\end{coro}
\begin{proof}
One can easily use the  estimate  in Proposition \ref{weyl law} to prove this, noting that the  $\Psi_{t,\eta_i}^{q(t)}$ corresponds to some metric $g(t)$ as in Proposition \ref{part 1 of mt} with $\eta_i$ given a priori, and the fact $\exp(-t^{-\frac{1}{n}})<t^l$ for any $l\geq 2$ as $t\rightarrow 0^+$.
\end{proof}
The truncated almost conformal embedding $\Psi^{q(t)}_{t,\eta_i}$ will be later perturbed to a family of honest conformal embeddings in Section 6. 

\section{G\"unther's iteration and the modification to conformal case} 
We may start by stating some conventions. Assume $u=(u_1,\cdots,u_N)\in C^\infty (M, \mathbb{R}^N)$ is a smooth embedding, and let the metric $g \in C^{2,\alpha}(M,\mathrm{Sym}^{ \otimes 2}T^*M)$. Then $\nabla u=(\nabla u_1,\cdots,\nabla u_N)\in C^\infty (M,T^*M\otimes \mathbb{R}^N)$ as a section, and we further denote $\nabla u\cdot \nabla u:=\nabla u_1\otimes \nabla u_1+\cdots +\nabla u_N\otimes \nabla u_N\in C^\infty (M,T^*M\otimes T^*M )$ as a $(0,2)$-tensor field. In local coordinates, $\nabla u$ may be identified with an $n\times N$ matrix, where $n=\dim M$, and $\nabla u\cdot \nabla u$ can be computed as $\nabla u \nabla u^T$. Eventually, as we will see in the later sections, $u$ is meant to represent $\Psi_t$ and $\Psi_{t,\eta_i}^{q(t)}$. 

\par 
Let's begin with the assumption that $u$  is given readily as an almost  conformal map, which satisfies
\begin{equation}\label{almost conformal equation}
    \nabla u\cdot \nabla u-\frac{\tr_g(\nabla u\cdot \nabla u)}{n}g=-f+\frac{\tr_g f}{n}g,
\end{equation} where $f$ is a `small' symmetric 2-tensor, and $\nabla$ is the Levi-Civita connection of $(M,g)$.  To attain our goal of this paper, it will suffice to find a map $u+v :M \longrightarrow \mathbb{R}^N$ that solves the equation
\begin{equation}\label{conformal want}
    \nabla(u+v)\cdot \nabla(u+v)-\frac{\tr_g(\nabla(u+v)\cdot \nabla(u+v))}{n}g=0.
\end{equation} 


Then, after subtract \eqref{conformal want} by \eqref{almost conformal equation},  our goal becomes finding a $v\in C^{r,\alpha}(M,\mathbb{R}^N)$, $r\geq 2$, satisfying the \textbf{conformal embedding equation}:
\begin{equation}\label{conformal embedding equation}
    \nabla u\cdot \nabla v-\frac{\mathrm{tr}_g(\nabla u\cdot \nabla v)}{n}g+\nabla v\cdot \nabla u-\frac{\mathrm{tr}_g(\nabla v\cdot \nabla u)}{n}g+\nabla v\cdot \nabla v-\frac{\mathrm{tr}_g(\nabla v\cdot \nabla v)}{n}g=f-\frac{\tr_g f}{n}g.
\end{equation}

\subsection{Free mappings.}  In this subsection, we would like to state the facts about \textbf{free mapping} and apply it to our conformal case. Throughout this paper, when we mention $C^{r,\alpha}$, $r\geq 2$, we will fix $0<\alpha<1$.  Also, since finding the optimistic $N$ is not the goal of this paper,  we could take $N$ to always be greater than or equal to $n+\frac{1}{2}n(n+1)$. 
\begin{defi}\label{def of P and P_c}
A $C^\infty$ embedding $u:M \longrightarrow \mathbb{R}^N$ is \emph{\textbf{free}} if, for each $x\in M$, the $n+\frac{1}{2}n(n+1)$ many vectors in $\mathbb{R}^N$:
$$\partial_iu(x), \partial_i\partial_ju(x), 1\leq i,j \leq n$$
form a $\mathrm{min}(N,n+\frac{1}{2}n(n+1))$-dimensional linear subspace of $\mathbb{R}^N$. Note that this definition is independent of the choice of coordinates. Denote such a subspace as $\mathrm{Span}\{\partial_iu(x), \partial_i\partial_ju(x)\}$. In this paper, we would also always denote a global linear operator  $\boldsymbol{P(u)}$ as follows:
\begin{equation*}
    \boldsymbol{P(u)}:=\begin{bmatrix}
      \nabla u\\
      \nabla\nabla u
    \end{bmatrix},
\end{equation*}
and another global operator $\boldsymbol{P_{c}(u)}$ that will be useful for conformal case:
\begin{equation*}
    \boldsymbol{P_c(u)}:=\begin{bmatrix}
      \nabla u\\
      \nabla\nabla u-\frac{\tr_g \nabla\nabla u}{n}g
    \end{bmatrix}.
\end{equation*}
\end{defi}
To clarify the definition, note that although for any point $x\in M$, $P$ could be thought of as a map $C^{\infty}(M)\longrightarrow T^*_x M \oplus\mathrm{Sym}^{\otimes 2} T_x^*M$, in our discussion, $P$ is always applied to a fixed free mapping $u$, which makes it a linear operator. The definition of $P(u)$ can be viewed as follows:
\begin{equation*}
    \begin{array}{ccccc}
         C^{\infty}(M,\mathbb{R}^N)&\overset{P|_x}{\xrightarrow{\hspace*{1cm}}} &\mathbb{R}^N\otimes(T^*_x M \oplus\mathrm{Sym}^{\otimes 2} T_x^*M) & \overset{\text{in normal coordinates}}{\xrightarrow{\hspace*{2.5cm}}}& \mathcal{L}(\mathbb{R}^N,\mathbb{R}^{n+\frac{n(n+1)}{2}})\\
         & & & &\\
         u& \xmapsto{\hspace*{0.6cm}} & \begin{bmatrix}
           \nabla u\\
           \nabla \nabla u
         \end{bmatrix}& \xrightarrow{\hspace*{1cm}}&  \begin{bmatrix}
           \nabla_i u_m\\
           \nabla_j \nabla_k u_m
         \end{bmatrix}_{\begin{array}{c}
             \scriptstyle 1\leq i,j,k \leq n, \\
              \scriptstyle 1\leq m\leq N
         \end{array}}
    \end{array}.
\end{equation*}

\begin{nota}\label{Notaion: large matrix}\normalfont
    In this paper, we would frequently encounter matrices or vectors with $n+\frac{n(n+1)}{2}$ rows, where each column represents an $n$-length vector (usually a gradient $\nabla$) and an $n\times n$ matrix (usually a hessian $\nabla \nabla$).  Our convention is as follows: we first place the $n$-length vector in the first $n$ rows; for the remaining $\frac{n(n+1)}{2}$ rows, we arrange the off diagonal entries of the $n\times n$ symmetric matrix in the first $\frac{n(n-1)}{2}$ rows, followed by the diagonal entries in the final $n$ rows.   \par
    For example, consider a free mapping $u=(u_1,\cdots , u_N)\in C^\infty (M,\mathbb{R}^N)$ in the normal coordinates  $\{x^i\}_{1\le i \le n}$ centered at $x\in M$. Then $P(u)(x)$ can be denoted as an $(n+\frac{n(n+1)}{2})\times N$ matrix that has a rank of $n+\frac{n(n+1)}{2}$, if $N\geq n+\frac{n(n+1)}{2}$:

\begin{equation*}
    P(u)(x)=\begin{bmatrix*}[l]
      \nabla_iu_1&  \nabla_iu_2  &\dots&\nabla_i u_N \,   \\
      \nabla_i\nabla_ju_1 & \nabla_i\nabla_ju_2 & \dots&\nabla_i\nabla_ju_N\,\\
      \nabla_k\nabla_ku_1 &\nabla_k\nabla_ku_2 & \dots&  \nabla_k\nabla_ku_N\,\\
    \end{bmatrix*}(x).
\end{equation*}
   $P_c(u)(x)$ has a similar matrix expression to $P(u)(x)$,differing only in the last $n$ rows, where each column is adjusted by subtracting $\frac{1}{n} \sum_{p=1}^n\nabla_p\nabla_p u_i$. \par

Here and in most discussions in this paper, it is enough to clarify the argument pointwisely, which allows us to pick the normal coordinates such that the Christoffel symbol vanishes at the point $x$.
\end{nota}
\begin{lemm}\label{Free mapping's linear algebra lemma} Let $r\geq 2$, for a free embedding $u\in C^\infty (M,\mathbb{R}^N)$,  and for $\xi\in C^{r,\alpha}(M,T^* M)$, $f\in C^{r,\alpha}(M,\mathrm{Sym}^{\otimes 2} T^*M)$, there exists a unique $v\in C^{r,\alpha}(M,\mathbb{R}^N)$ such that
\begin{equation}\label{2.5}
    P(u)\cdot v=\begin{bmatrix}
      \nabla u\\
      \nabla\nabla u   
    \end{bmatrix}v=\begin{bmatrix}
      \xi\\
      f
    \end{bmatrix}, \text{ and } v(x)\perp \mathrm{Ker}P(u)(x).
\end{equation}

\end{lemm}
\begin{proof} By the definition of $u$ being a free mapping of $u$, for each point $x\in M$, the $P(u)(x)$ is of full rank, hence $P(u)(x): \mathbb{R}^N \longrightarrow T^*_x M\oplus \mathrm{Sym}^{\otimes 2}T^*_xM$ is surjective, therefore the solution $v\in \mathbb{R}^N$ exists. After forcing $v(x)\perp \mathrm{Ker}P(u)(x)$, such $v(x)$ is unique.\par
 Notice that the pointwise solution for $v$ will yield the globally defined $v\in C^{r,\alpha}(M,\mathbb{R}^N)$ since  $P(u)$, $\xi$, $f$ are globally defined and  $\xi$ and $f$ are of $C^{r,\alpha}$.

\end{proof}

\begin{prop}\label{P_c ker}
Let $r\geq 2$, for a free embedding $u\in C^\infty (M,\mathbb{R}^N)$,  and for $\xi\in C^{r,\alpha}(M,T^* M)$, $f\in C^{r,\alpha}(M,\mathrm{Sym}^{\otimes 2} T^*M)$, there exists a unique $v\in C^{r,\alpha}(M,\mathbb{R}^N)$ satisfying the following equation:
\begin{equation}\label{2.6}
    P_c(u)\cdot v=\begin{bmatrix}
      \nabla u\\
      \nabla\nabla u-\frac{\tr_g(\nabla\nabla u)}{n}g
    \end{bmatrix}v=\begin{bmatrix}
      \xi\\
      f-\frac{\tr_g f}{n}g
    \end{bmatrix}, \text{ and } v(x)\perp \mathrm{Ker}P_c(u)(x).
\end{equation}
Moreover on each point $x\in M$, we have  $\mathrm{rank}(P_c(u)(x)) =\dim(\mathrm{Im} P(u)(x))=n+\frac{n(n+1)}{2}-1$,  and $\dim(\mathrm{Ker}P_c(u)(x))=\dim(\mathrm{Ker}P(u)(x))+1$.
\end{prop}
\begin{proof}
We will first show the existence and uniqueness of the solution at a fixed point $x\in M$. First notice that,  given any $\xi$ and $f$, the unique  $v_0$ such that
\begin{equation}\label{2.7}
  P(u)\cdot v_0=\begin{bmatrix}
      \nabla u\\
      \nabla\nabla u   
    \end{bmatrix}v_0=\begin{bmatrix}
      \xi\\
      f
    \end{bmatrix}, \text{ and } v_0(x)\perp \mathrm{Ker}P(u)(x)    
\end{equation}
also satisfies
\begin{equation*}
     P_c(u)\cdot v_0=\begin{bmatrix}
      \nabla u\\
      \nabla\nabla u-\frac{\tr_g(\nabla\nabla u)}{n}g
    \end{bmatrix}v_0=\begin{bmatrix}
      \xi\\
      f-\frac{\tr_g f}{n}g
    \end{bmatrix}.
\end{equation*} 
Indeed, this holds due to the linearity of  $\tr_g$:
\begin{equation*}
\langle \frac{\tr_g(\nabla\nabla u)}{n}g,v_0\rangle=\sum_{j=1}^N\frac{\tr_g(\nabla\nabla u_j\cdot (v_0)_j)}{n}g=\frac{\tr_g(\sum_{j=1}^N\nabla\nabla u_j\cdot (v_0)_j)}{n}g=\frac{\tr_g f}{n}g.
\end{equation*}\par
Hence for any  $v$ satisfies $P_c(u)\cdot v=[\xi, f-\frac{\tr_g f}{n}g]^T$, the  $v$ has to be in the form  $$v=v_0+w,$$ where $v_0$ is the unique vector attained by $\eqref{2.7}$, and $w$ is an arbitrary vector in $\mathrm{Ker}P_c(u)(x)$. Moreover by forcing  $v(x)=v_0+w\in \mathrm{Ker}^\perp P(u)(x)$,  we will get a unique $w\in \mathrm{Ker}P_c(u)(x)$. By the same reason as in former lemma, the global solution $v\in C^{r,\alpha}(M,\mathbb{R}^N)$ also exists and is unique.\par

\par 
Express $P_c(u)(x)$ in normal coordinates as an $\frac{n(n+3)}{2}\times N$ matrix as in Notation \ref{Notaion: large matrix}. Summing the last $n$ rows we have

 \begin{equation*}
 \sum_{i=1}^n\Big(\nabla_i\nabla_iu(x)-\frac{\sum_{p=1}^n(\nabla_p\nabla_p u(x))}{n}\Big)=0.
 \end{equation*}
Thus $\rank (P_c(u)(x))\leq \frac{n(n+3)}{2}-1$. Also notice that by direct summing with a one dimensional space, we have: $$\mathrm{Span} \{\nabla_iu(x),\nabla_i\nabla_ju(x)-\frac{\sum_{p=1}^n(\nabla_p\nabla_pu(x))}{n}\}\oplus\{\frac{\sum_{p=1}^n(\nabla_p\nabla_pu(x))}{n}\}\supseteq \mathrm{Span} \{\nabla_iu(x),\nabla_i\nabla_ju(x)\}.$$ 
Therefore  $\rank(P_c(u)(x))=\dim(\mathrm{Im}P_c(u)(x))\geq \frac{n(n+3)}{2}-1$, where the conclusion arrives.
\end{proof}

\begin{rema} \label{essenttial remark}
The former proposition states the decomposition that $$\mathrm{Ker}P_c(u)(x)=\mathrm{Ker}P(u)(x)\oplus \{w(x)\}$$ for each point $x\in M$ where $w\in \mathbb{R}^N$. Moreover, we can describe the generator $w(x)$ precisely  here. Let $w\in C^{r,\alpha}(M,\mathbb{R}^N)$ be the unique one such that
\begin{equation*}
    P(u)\cdot w=\begin{bmatrix}
      0\\
      g
    \end{bmatrix},\quad \text{and } w(x)\perp \mathrm{Ker}P(u)(x). 
\end{equation*}
By Lemma \ref{Free mapping's linear algebra lemma}, such  $w$ exists and is unique. Consequently $P_c(u)\cdot w=0$, i.e. $w\in \mathrm{Ker}P_c(u)(x)$. By the definition of $w$, especially that $w(x)\perp \mathrm{Ker}P(u)(x)$, we see this $w$ is exactly the one in the decomposition.\par

 Hence for all the $v\in C^{r,\alpha}(M,\mathbb{R}^N)$ that satisfies $v(x)\perp \mathrm{Ker}P(u)(x)$ and solves $P_c(u)\cdot v=[
      f \quad
      h-\frac{\tr_g h}{n}g
    ]^T$, it has to be in the form that 
    \begin{equation*}
        v= v_0+k\cdot w,
    \end{equation*}
    where $v_0$ is the unique solution of $P(u)\cdot v_0=[
      f \quad
      h
    ]^T$, $v_0(x)\perp \mathrm{Ker}P(u)(x)$, and $w$ is the unique vector defined above, $k\in C^{r,\alpha}(M,\mathbb{R})$.\par 
Even though it is not closely related to our later goal, it is worth noting that by further requiring $v(x)\perp \mathrm{Ker}P_c(u)(x) \supset \mathrm{Ker}P(u)(x)$, we will have a unique solution $v(x)$. In fact, we have 
\begin{equation*}
    v=v_0-\frac{\langle v_0, w \rangle_{\mathbb{R}^N}}{\langle w, w\rangle_{\mathbb{R}^N}}w.
\end{equation*} Notice that $\langle w, w\rangle_{\mathbb{R}^N}(x)\neq 0$ for any $x\in M$, so the expression  is a well defined global one. Also, notice that even though $\langle h-\frac{\tr_gh }{n}g, g\rangle =0$, the inner product $\langle v_0, w \rangle_{\mathbb{R}^N}$ isn't always equal to zero.
\end{rema}

\subsection{G\"unther's lemma in conformal case.}\label{Gunther}  In this subsection, we would follow G\"unther \cite{G1} to perform a detailed computation, summarizing the results as a lemma at the end. Again, the well-known Einstein summation notation is employed throughout.\par
Let $\nabla $ be the Levi-Civita connection of $(M,g)$. The Laplacian in use  is the connection Laplacian, $\Delta:=\mathrm{tr}\nabla^2$, applicable to all the functions and tensors of at least $C^2$ smooth. The Ricci curvature is defined as $R_{ik}:=R_{lik}{}^l$ in our convention. \par
In our setting, the eigenvalues are  the $\lambda_i$'s satisfying  $\Delta f +\lambda f=0$, hence $0<\lambda_1\leq \lambda_2\leq \cdots $. Thus for some positive constant number $\epsilon$, the $\Delta-\epsilon$ is an isomorphism between $C^{r,\alpha}$ and $C^{r-2,\alpha}$ of functions and tensors  of various sizes on compact $M$, for $r\geq 2$. Additionally, it has a unique inverse denoted as $(\Delta-\epsilon)^{-1}$. For our application, we shall take $\epsilon=1$ throughout.

Let $u\in C^\infty (M, \mathbb{R}^N)$ be a free embedding approximating the given smooth metric $g$ as in \eqref{almost conformal equation}. Let $r\geq 2$, our goal is to find $v\in C^{r,\alpha}(M,\mathbb{R}^N)$ satisfying \eqref{conformal embedding equation}. \par
We first consider the term with trace:
   $\nabla u\cdot \nabla v+\nabla v\cdot \nabla u +\nabla v\cdot \nabla v$, apply $\Delta-1$ to it and set it equal to $f$:
   \begin{equation*}
    (\Delta-1)(\nabla u\cdot \nabla v)+(\Delta-1)(\nabla v\cdot \nabla u)+(\Delta-1)(\nabla v\cdot \nabla v)=(\Delta-1)f,
\end{equation*}
where the third derivatives of $v$ are to be considered as distributions if $r<3$, and this won't affect the following computation.
\par
Careful treatment is needed by  using a local coordinates, denoted as  $\{x^i\}_{1\le i \le n}$, which are not necessarily normal coordinates. The first term to compute is the quadratic term $(\Delta-1)(\nabla_iv\cdot \nabla_j v)$ involving $v$:

\begin{equation*}
    \begin{aligned}
    (\Delta&-1)(\nabla_iv \,dx^i\cdot\nabla_jv \,dx^j)\\
    =&\nabla^l(\nabla_l(\nabla_iv \,dx^i\cdot\nabla_jv\,dx^j))-\nabla_iv\cdot\nabla_jv\,dx^i\otimes dx^j\\
    =&\Delta(\nabla_iv\,dx^i)\cdot(\nabla_jv\,dx^j)+(\nabla_iv\,dx^i)\cdot\Delta(\nabla_jv\,dx^j)\\
    &+2\nabla^l(\nabla_iv\,dx^i)\cdot\nabla_l(\nabla_jv\,dx^j)-\nabla_iv\cdot\nabla_jv\,dx^i\otimes dx^j\\
    =&\nabla_i(\Delta v)\,dx^i\cdot\nabla_jv\,dx^j+R_i{}^k\nabla_kv \,dx^i\cdot\nabla_jv\,dx^j+\nabla_iv\,dx^i\cdot\nabla_j(\Delta v)\,dx^j\\
   \, &+\nabla_iv \,dx^i\cdot R_j{}^k\nabla_kv\,dx^j+2\nabla^l(\nabla_iv\,dx^i)\cdot\nabla_l(\nabla_jv\,dx^j)-\nabla_iv\cdot\nabla_jv\,dx^i\otimes dx^j\\
    =&\nabla_i(\Delta v\cdot\nabla_j v\,dx^j)dx^i-\Delta v\cdot\nabla_i(\nabla_j v
    \,dx^j) dx^i+
    \nabla_j(\nabla_iv\,dx^i\cdot\Delta v)dx^j-\nabla_j(\nabla_iv\,dx^i)dx^j\cdot \Delta v \\
    \,&+(R_i{}^k\nabla_jv+R_j{}^k\nabla_iv)\cdot \nabla_kv\,dx^i\otimes dx^j
    +2\nabla^l(\nabla_iv\,dx^i)\cdot\nabla_l(\nabla_jv\,dx^j)-\nabla_iv\cdot\nabla_jv\,dx^i\otimes dx^j\\
    =&2L_{ij}(v,v) dx^i\otimes dx^j+\nabla_i(\Delta v\cdot\nabla_j v\,dx^j)dx^i+\nabla_j(\nabla_iv\,dx^i\cdot\Delta v)dx^j,
\end{aligned}
\end{equation*}
here for brevity, we denoted 

\begin{align}
    L_{ij}(v,v)\,dx^i\otimes dx^j:=&\nabla^l(\nabla_iv\,dx^i)\cdot\nabla_l(\nabla_jv\,dx^j)-\Delta v\cdot\nabla_i(\nabla_jv\,dx^j)dx^i\notag\\
    &+\{-\frac{1}{2}\nabla_iv \cdot\nabla_jv+\frac{1}{2}(R_i{}^k\nabla_jv+R_j{}^k\nabla_iv)\cdot \nabla_kv\}dx^i\otimes dx^j.\label{2.16}
\end{align}
For the other terms involving $u $ and $v$, we get the following by switching the Laplacian and the covariant derivative:
\begin{equation*}
\begin{aligned}
    (\Delta&-1)(\nabla_iu dx^i\cdot \nabla_jv dx^j)\\
    =&(\Delta-1)(\nabla_j(\nabla_iu dx^i\cdot v)dx^j)-(\Delta-1)(\nabla_j(\nabla_iu dx^i)dx^j\cdot v)\\
    =&\nabla_j((\Delta-1)(\nabla_iu\,dx^i\cdot v))dx^j+\{2R_j{}^k{}_i{}^n\nabla_k(\nabla_n u\cdot v) +R_{j}{}^k{}_i{}^n(\nabla_m u\cdot v)(-\Gamma_{kn}^m)+\nabla^k({R_{jki}}^n)(\nabla_n u\cdot v)\\
   &+g^{kl}{R_{jmi}}^n(\nabla_n u\cdot v) (-\Gamma_{lk}^m)+R_j{}^n\nabla_n(\nabla_i u\cdot v) +R_j{}^n (\nabla_m u\cdot v) (-\Gamma_{ni}^m)\}dx^i\otimes dx^j
   \\&-(\Delta-1)(\nabla_j\nabla_iu\cdot v\,dx^i\otimes dx^j)
   -(\Delta-1)(\nabla_nu\,\Gamma_{ji}^n\cdot v\,dx^i\otimes dx^j).
\end{aligned}
\end{equation*}
Similar computation for the other one, 
\begin{equation*}
    \begin{aligned}
(\Delta&-1)(\nabla_ju\,dx^j\cdot \nabla_iv\,dx^i)\\=&\nabla_i((\Delta-1)(\nabla_ju\,dx^j\cdot v))dx^i+\{2R_i{}^k{}_j{}^n\nabla_k(\nabla_n u\cdot v) +R_{i}{}^k{}_j{}^n(\nabla_m u\cdot v)(-\Gamma_{kn}^m)+\nabla^k({R_{ikj}}^n)(\nabla_n u\cdot v)\\ &+g^{kl}{R_{imj}}^n(\nabla_n u\cdot v) (-\Gamma_{lk}^m)+R_i{}^n\nabla_n(\nabla_j u\cdot v) +R_i{}^n (\nabla_m u\cdot v) (-\Gamma_{nj}^m)\}dx^i\otimes dx^j\\
&-(\Delta-1)(\nabla_i\nabla_ju\cdot v\, dx^i\otimes dx^j)-(\Delta-1)(\nabla_n u\, \Gamma_{ij}^n\cdot v\, dx^i\otimes dx^j).
    \end{aligned}
\end{equation*}
We could denote the following notion of $r_{ij}^n$, for $w=w_ndx^n\in C^{r,\alpha}(M,T_x^*M)$:
\begin{equation}\label{2.19}
\begin{aligned}
r_{ij}^nw_ndx^i\otimes dx^j:=&\{
2R_i{}^k{}_j{}^n\nabla_kw_n +R_{i}{}^k{}_j{}^mw_n(-\Gamma_{km}^n)+\nabla^k({R_{ikj}}^n)w_n \\
&+g^{kl}{R_{imj}}^nw_n (-\Gamma_{lk}^m)+R_i{}^m w_n (-\Gamma_{mj}^n)\}dx^i\otimes dx^j.
\end{aligned}
\end{equation}\par
Combining everything, we get: 
\begin{equation}\label{vanishing terms}
\begin{aligned}
&(\Delta-1)(\nabla_iu\cdot\nabla_jv\,dx^i\otimes dx^j+\nabla_ju\cdot\nabla_iv\,dx^i\otimes dx^j+\nabla_iv\cdot\nabla_jv \, dx^i\otimes dx^j)\\
=& \nabla_j\{(\Delta-1)(\nabla_iu\,dx^i\cdot v)+\nabla_iv\,dx^i\cdot\Delta v\}dx^j+\nabla_i\{(\Delta-1)(\nabla_ju\,dx^j\cdot v)+\Delta v\cdot\nabla_j v\,dx^j\}dx^i\\
&+\{2L_{ij}(v,v)+r^n_{ij}(\nabla_n u\cdot v)+R_i^n\nabla_n(\nabla_ju \cdot v)+r^n_{ji}(\nabla_n u\cdot v)+R_j^n\nabla_n(\nabla_iu \cdot v)\}dx^i\otimes dx^j\\
&-2(\Delta-1)(\nabla_i(\nabla_ju\,dx^j)dx^i\cdot v).
\end{aligned}
\end{equation}
In \cite{G1},  G\"unther made the key observation that if  $\nabla_i u \cdot v dx^i$ equal to $-(\Delta-1)^{-1}\{\Delta v\cdot \nabla_iv\,dx^i\}$, multiple benefits will arise, including the ability to \textbf{avoid loss of differentiability}. In fact, if we force $\nabla_i u \cdot v dx^i=-(\Delta-1)^{-1}\{\Delta v\cdot \nabla_iv\,dx^i\}$, the equation
\begin{equation*}
    f=\nabla  v \cdot\nabla u+\nabla u \cdot\nabla v +\nabla v\cdot \nabla v
\end{equation*}
can be solved by the $v$ that satisfies the following:
\begin{equation}\label{defi of u*v}
    \begin{aligned}
        \nabla_i u \,dx^i \cdot v 
        =&-(\Delta-1)^{-1}\{\Delta v\cdot \nabla_iv\,dx^i\},\\
        \nabla_i(\nabla_j u \,dx^j) dx^i\cdot v =& \frac{1}{2}(\Delta-1)^{-1}(\{2L_{ij}(v,v)+(r^n_{ij}+r^n_{ji})(\nabla_n u\cdot v )\\
        &+R_j{}^n\nabla_n(\nabla_i u\cdot v )+R_i{}^n\nabla_n(\nabla_ju \cdot v)\}dx^i\otimes dx^j)-\frac{1}{2}f_{ij}dx^i\otimes dx^j\\
       = & \frac{1}{2}(\Delta-1)^{-1}\Big([2L_{ij}(v,v)+(r_{ij}^n+r_{ji}^n)(-(\Delta-1)^{-1}\{\Delta v\cdot \nabla v\})_n\\
        &+R_j{}^n\nabla_n((\Delta-1)^{-1}\{\Delta v\cdot \nabla v\})_i+R_i{}^n\nabla_n((\Delta-1)^{-1}\{\Delta v\cdot \nabla v\})_j]dx^i\otimes dx^j \Big)\\
&-\frac{1}{2}f_{ij}dx^i\otimes dx^j.
    \end{aligned}
\end{equation}
Here, the notation $(\quad )_j$ denotes the coefficient of the 1-form inside it with respect to $dx^j$, i.e., $w=(w)_jdx^j$ for any 1-form $w$.  Note that, aside from $f$, the right-hand side of \eqref{defi of u*v} does not involve $u$, is quadratic about $v$, and \textbf{does not lose the differentiability of $v$,} i.e., if $v\in C^{r,\alpha}(M,\mathbb{R}^N)$ then the terms are in $C^{r,\alpha}(M,\mathrm{Sym}^{\otimes 2} T_x^*M)$. Then, define  $Q_u(v,v)\in C^{r,\alpha}(M,\mathbb{R}^N)$ as the unique solution (thanks to Lemma \ref{Free mapping's linear algebra lemma}) of the following:
\begin{equation}\label{defi of Q(v,v)}
    \begin{aligned}
        \nabla u\cdot Q_u(v,v) 
        =&-(\Delta-1)^{-1}\{\Delta v\cdot \nabla v\}, \\
        \nabla\nabla u\cdot Q_u(v,v) =& M_{ij}(v),
   \end{aligned}
\end{equation}
        \begin{align*}
  \text{where }  \quad M_{ij}(v):=&\frac{1}{2}(\Delta-1)^{-1}([2L_{ij}(v,v)+(r_{ij}^n+r_{ji}^n)(-(\Delta-1)^{-1}\{\Delta v\cdot \nabla v\})_n\\
        +R_j{}^n&\nabla_n((\Delta-1)^{-1}\{\Delta v\cdot \nabla v\})_i+R_i{}^n\nabla_n((\Delta-1)^{-1}\{\Delta v\cdot \nabla v\})_j]dx^i\otimes dx^j ).       
        \end{align*}

By the freeness of $u$ and by applying Lemma \ref{Free mapping's linear algebra lemma},  $Q_u(v,v)$ exists and is unique if we require that $Q_u(v,v)(x)\perp \mathrm{Ker}P(u)(x)$ for any $x\in M$. Hence \eqref{vanishing terms} can be rewritten as
\begin{equation*}
\begin{aligned}
    &(\Delta -1)(\nabla u\cdot \nabla v+\nabla v\cdot \nabla u+ \nabla v \cdot \nabla v)\\
    =&2\mathrm{Sym}(\nabla\{(\Delta-1)(\nabla u\cdot (v-Q_u(v,v)))\})-2(\Delta-1)\{\nabla\nabla u\cdot (v-Q_u(v,v))\},
\end{aligned}
\end{equation*}
where $\mathrm{Sym}$ is the \textit{symmetrization} of a 2-tensor. Using the notation $\mathrm{tr}^\perp_g$ in Remark \ref{denotion of C}, the  conformal equation (\ref{conformal embedding equation})  can be rewritten as
\begin{equation*}
   \mathrm{tr}_g^\perp(f)= \mathrm{tr}_g^\perp\Big(2(\Delta-1)^{-1}\mathrm{Sym}(\nabla\{(\Delta-1)(\nabla u\cdot (v-Q_u(v,v)))\})-2\nabla\nabla u\cdot (v-Q_u(v,v))
   \Big).\end{equation*}
\par
To summarize, we have the conformal version of G\"unther's lemma.
\begin{lemm}[G\"unther's Lemma with conformal operator] Let $u\in C^\infty(M,\mathbb{R}^N)$ be a free embedding. Then the following conformal embedding equation \eqref{conformal embedding equation} that aims to find $v\in C^{r,\alpha}(M,\mathbb{R}^N), \, r\geq 2$
\begin{equation*}
    \mathrm{tr}_g^\perp(\nabla u\cdot \nabla v)+\mathrm{tr}_g^\perp(\nabla v\cdot \nabla u)+\mathrm{tr}_g^\perp (\nabla v\cdot \nabla v)=\mathrm{tr}_g^\perp(f)
\end{equation*}
is equivalent to
\begin{equation}\label{conformal gunther equation}
    \mathrm{tr}_g^\perp(f)= \mathrm{tr}_g^\perp\Big(2(\Delta-1)^{-1}\mathrm{Sym}(\nabla\{(\Delta-1)(\nabla u\cdot (v-Q_u(v,v)))\})-2\nabla\nabla u\cdot (v-Q_u(v,v))
   \Big),
\end{equation}
   where  $Q_u(v,v)\in C^{r,\alpha}(M,\mathbb{R}^N),\, r\geq 2 $ is  quadratic about $v$,  as defined in \eqref{defi of Q(v,v)}.
\end{lemm}

\begin{rema}[Nash's simplification]\label{Nash's simplification}
    In \cite{N2}, one of Nash's tricks is to force $\nabla u \cdot v=0$ to simplify the original local isometric immersion equation, see equation (B3) in \cite{N2}. G\"unther's trick in \cite{G1}, instead, is to force $\nabla u \cdot (v-Q_u(v,v))=0$, thus \eqref{conformal gunther equation} will become 
    \begin{equation}
        P_c(u)\cdot (v-Q_u(v,v))=\begin{bmatrix}
            0\\
            \frac{-(f-\frac{\mathrm{tr}_g f}{n}g)}{2}
        \end{bmatrix},
    \end{equation}
    where $P_c(u)$ is as defined in Definition \ref{def of P and P_c}.
\end{rema}
\section{The properties of $P_c$, and its right inverses}
\subsection{Singularity of $P_c(\Psi_t)P_c^T(\Psi_t)$.}
In Proposition \ref{P_c ker}, for any free embedding $u\in C^\infty (M,\mathbb{R}^N)$, we have seen that $P_c(u)$ is of rank $\frac{n(n+3)}{2}-1$ in normal coordinates at a point $x\in M$. Hence, naturally, $P_c(u)P_c^T(u)(x)$ is singular and of rank $\frac{n(n+3)}{2}-1$. \par

To apply this to our case, we need to show that the functions that of interest are free mappings. It is established in \cite[Theorem 18]{WZ} that $\Psi_t:M\longrightarrow \ell^2$ is a free mapping by expanding the heat kernel. In Corollary \ref{coro on freeness of truncated}, we will justify the freeness of the truncated map $\Psi_{t,\eta_i}^{q(t)}: M\longrightarrow \mathbb{R}^{q(t)}$ as defined in Definition \ref{defi of truncation}.\par
For additional insight, in this subsection, we  will  explicitly present the matrix expression of $P_c(\Psi_t)P_c^T(\Psi_t)$  and compute its rank in proposition \ref{P_cP_c^T}. 
Here with respect to normal coordinates in the neighbourhood of $x\in M$, $P(\Psi_t)(x)$ and $P_c(\Psi_t)(x)$ are taken as $\frac{n(n+3)}{2}\times \infty$ matrices for $\Psi_t:M\longrightarrow \ell^2$, while $P(\Psi_{t,\eta_i}^{q(t)})(x)$ and $P_c(\Psi_{t,\eta_i}^{q(t)})(x)$ are taken as $\frac{n(n+3)}{2}\times \infty$ matrices.  \par

We need a proposition on linear algebra first. In the remainder of this paper, $J_n$ denotes  an $n\times n$ matrix with all entries equal $1$.
\begin{prop}\label{Xi}
Let $\sigma\in (-\frac{1}{n-1},1)$. Then the $n\times n$ matrix
\begin{equation}
    \Xi_n(\sigma):=[\theta_{ij}]_{1\leq i,j\leq n}
\end{equation}
with $\theta_{ii}=1$ and $\theta_{ij}=\sigma$ $(i\neq j)$ is invertible. And the condition for $\sigma> -\frac{1}{n-1}$ is sharp, more precisely, 
    $\Xi_n(-\frac{1}{n-1})$ is not invertible and of rank $n-1$.

\end{prop}
\begin{proof}
The invertibility of $\Xi_n(\sigma)$ when $-\frac{1}{n-1}<\sigma<1$ is due to \cite[Corollary 26]{WZ}. We only need to verify the rank of $\Xi_n(-\frac{1}{n-1})$ is $n-1$. Indeed, let $J_n$ be an $n\times n$ matrix consisting all $1$'s,
\begin{equation*}
    (n-1)\cdot \Xi_n(-\frac{1}{n-1})=  \begin{bmatrix}
(n-1) & -1  & \dots&-1\\
-1& (n-1)  & \dots&-1\\
\cdots\\
-1& -1 &  \cdots&(n-1)\\
\end{bmatrix}=nI_n-J_n,
\end{equation*}
  it can be easily seen that it is of rank $n-1$.\end{proof}

\begin{prop}\label{P_cP_c^T}
 With respect to normal coordinates in the neighbourhood of $x\in M$, the 
$\frac{n(n+3)}{2}\times \infty$ matrix $P_c(\Psi_t)(x)$ can be expressed in the following way:
\begin{equation}
P_c(\Psi_t)(x)
=[\nabla_i\Psi_t(x)\quad \nabla_i\nabla_j\Psi_t(x) \quad \nabla_k\nabla_k\Psi_t(x)-\frac{\sum_p \nabla_p\nabla_p\Psi_t}{n}(x)]^T
\end{equation}
 $i\neq j$, $1\leq i,j,k,p\leq n$, then locally we have the following when $t\rightarrow 0_+$:
\begin{equation}\label{exp P_cP_cT}
    \begin{aligned}
    P_c&(\Psi_t)P_c^T(\Psi_t)(x)\\
    &=\begin{bmatrix}
I_n+O(t) &  O(t)\\
O(t) & \frac{1}{2t}\cdot \begin{pmatrix}
\begin{bmatrix}
  I_{\frac{n(n-1)}{2}} & 0\\ 
  0 & 3 \cdot \Xi(\frac{1}{3})-\frac{n+2}{n}\cdot J_n
\end{bmatrix}
+O(t)
\end{pmatrix}
\end{bmatrix}.
    \end{aligned}
\end{equation}
Furthermore, the $n \times n$ matrix
  \begin{equation}\label{rank of M}
3\cdot \Xi(\frac{1}{3})-\frac{n+2}{n}\cdot J_n
\end{equation}
is of rank $\frac{n(n+1)}{2}-1$.

\end{prop}
\begin{proof}
    The expression (\ref{exp P_cP_cT}) can be achieved by employing the formulas in \cite[Proposition 21]{WZ} and direct computation, which we omit here. \par
    To see the rank of (\ref{rank of M}), notice that $$3\cdot \Xi(\frac{1}{3})-\frac{n+2}{n}\cdot J_n= \frac{2n-2}{n}\cdot \Xi(-\frac{1}{n-1}).$$  By  Proposition \ref{Xi}, we know $\Xi(-\frac{1}{n-1})$  is not invertible and of rank $n-1$. 
\end{proof}
The following proposition is due to \cite[Corollary 29]{WZ}. It is interesting to compare it with the case of $P_c(\Psi_t)P_c^T(\Psi_t)$ described in Proposition \ref{P_cP_c^T}.
\begin{prop}
    \label{construction of PPT}
For each point $x\in M$, and with respect to  normal coordinates  around $x$, we have: 
\begin{equation}\label{expression of PP^T}
 \begin{aligned}
    P&(\Psi_t)P^T(\Psi_t)(x)\\
    &=\begin{bmatrix}
I_n+O(t) &   O(t)\\
 O(t) & \frac{1}{2t}\cdot 
\begin{bmatrix}
  I_{\frac{n(n-1)}{2}} & 0\\ 
  0 & 3\cdot \Xi(\frac{1}{3})
\end{bmatrix}
+O(t)

\end{bmatrix}.
    \end{aligned}
\end{equation}
Moreover, this $P(\Psi_t)P^T(\Psi_t)(x)$ is invertible.
\end{prop}

\begin{coro}\label{coro on freeness of truncated}
    Given $\eta_i$ and let $t\in(0,t_0]$ as in Theorem \ref{main thm 1.1} and \ref{mainthm 1.3}. For $q(t)\geq t^{-\frac{n}{2}-1}$, 
    the truncated mapping $\Psi_{t,\eta_i}^{q(t)}: M\longrightarrow \mathbb{R}^{q(t)}$ is a \emph{free mapping}. Moreover, $P(\Psi_{t,\eta_i}^{q(t)})P^T(\Psi_{t,\eta_i}^{q(t)})(x)$ and $P_c(\Psi_{t,\eta_i}^{q(t)})P_c^T(\Psi_{t,\eta_i}^{q(t)})(x)$ have the same expression with $P(\Psi_t)P^T(\Psi_t)(x)$ and $P_c(\Psi_t)P_c^T(\Psi_t)(x)$ as in  \eqref{expression of PP^T} and \eqref{exp P_cP_cT}, respectively.
\end{coro}
\begin{proof}
 By Proposition \ref{weyl law}, for given $\eta_i$, we have 
\begin{equation*}
\begin{aligned}
    &\lVert P_c(\Psi_{t,\eta_i}^{q(t)})P_c^T(\Psi_{t,\eta_i}^{q(t)})(x) -P_c(\Psi_t)P_c^T(\Psi_t)(x)\rVert_{C^r(M)} \\
    \leq& 2\sum_{|\Vec{\gamma}|\leq2,|\Vec{\beta}|\leq2} |\sum_{j\geq q(t)+1} e^{-\lambda_jt}D^{\Vec{\gamma}}\phi_jD^{\vec{\beta}}\phi_j| \leq C e^{(-t^{-\frac{1}{n}})}
\end{aligned}
\end{equation*}
 for any $r\geq 0$, where we are taking the above matrix norm $\lVert \,\cdot\, \rVert_{C^r}$ as the maximum of each entry's $C^r$ norm. Therefore $P_c(\Psi_{t,\eta_i}^{q(t)})P_c^T(\Psi_{t,\eta_i}^{q(t)})(x)$ has the same expression as in \eqref{exp P_cP_cT}. Similarly, $P(\Psi_{t,\eta_i}^{q(t)})P^T(\Psi_{t,\eta_i}^{q(t)})(x)$ has the same expression as in  \eqref{expression of PP^T}.\par
 Consequently, $P(\Psi_{t,\eta_i}^{q(t)})P^T(\Psi_{t,\eta_i}^{q(t)})(x)$ is nonsingular for each point $x\in M$ as an $\frac{n(n+3)}{2}\times \frac{n(n+3)}{2}$ matrix. Thus we obtain the freeness of $\Psi_{t,\eta_i}^{q(t)}$.
\end{proof}

\subsection{Right inverses of $P_c$.}\label{sec: E}

In the last subsection, we see that the $P_c(\Psi_{t,\eta_i}^{q(t)})P_c^T(\Psi_{t,\eta_i}^{q(t)})(x)$ is of rank $\frac{n(n+3)}{2}-1$. This implies that we cannot expect to find a right inverse operator of $P_c(\Psi_{t,\eta_i}^{q(t)})$ for arbitrary right-hand vectors, but still can find the right inverses of $P_c(\Psi_{t,\eta_i}^{q(t)})$ for the right-hand ones in the image of $P_c(\Psi_{t,\eta_i}^{q(t)})$. This illustrates  the difference between the local conformal embedding question and the local isometric one.\par
Recall in Corollary \ref{coro 2.10}, the remainder term $O(t^l)$ is a symmetric 2-tensor subtracting its own trace, we denote it as $h$, which corresponds to the small difference term $f-\frac{\tr_g f}{n}g$ in the conformal embedding equation \eqref{conformal embedding equation}. Additionally, let $G\subset \mathrm{Sym}^{\otimes 2}T^*M$ be the subbundle defined by:

\begin{equation}\label{definition of bundle G}
    G_x:=\{ s_x\in \mathrm{Sym}^{\otimes 2}T^*M
 \big|\,  \mathrm{tr}_{g(x)}s_x=0\}.
\end{equation} Thus, $h\in C^{r,\alpha}(M,G) $, and $P_c(\Psi_{t,\eta_i}^{q(t)})$ becomes surjective onto $C^{r,\alpha}(M,G) $, thus has right inverse.

Our goal now is to explicitly construct a family of solutions of the following equation
\begin{equation*}
    P_c(\Psi_{t,\eta_i}^{q(t)})\cdot v=\begin{bmatrix}
        0 \\ h
    \end{bmatrix},
\end{equation*}
where we are solving for $v\in C^{r,\alpha}(M,\mathbb{R}^{q(t)}), \text{ and } v(x)\perp \mathrm{Ker}P(\Psi_{t,\eta_i}^{q(t)})(x)$.
Note that it is a version of \eqref{2.6} with $\Psi_{t,\eta_i}^{q(t)}$ involved.

\begin{theo} \label{construction of E_c}
For $q=q(t)\geq t^{-\frac{n}{2}-1} $, assume that $\Psi_{t,\eta_i}^{q(t)}\in C^\infty(M,\mathbb{R}^{q(t)})$ is defined as before and define the traceless 2-tensor bundle $G$ as in \eqref{definition of bundle G}.\par
Then for $h\in C^{r,\alpha}(M,G)$, $v\in C^{r,\alpha}(M,\mathbb{R}^q), \text{ and } v(x)\perp \mathrm{Ker}P(\Psi_{t,\eta_i}^{q(t)})(x)$,
we have
\begin{equation}
    E(\Psi_{t,\eta_i}^{q(t)})(0,h)+kE(\Psi_{t,\eta_i}^{q(t)})(0,g)=v \qquad \Longleftrightarrow \qquad \begin{bmatrix}0 \\h\end{bmatrix}=P_c(\Psi_{t,\eta_i}^{q(t)}) \cdot v
\end{equation} for some $k\in C^{r,\alpha}(M,\mathbb{R})$, where $E(\Psi_{t,\eta_i}^{q(t)}):C^{r,\alpha}(M,T^*M)\times C^{r,\alpha}(M,G)\longrightarrow C^{r,\alpha}(M,\mathbb{R}^q),\, r\geq 2$ is the right inverse of $P(\Psi_{t,\eta_i}^{q(t)})$ defined as  
$$E(\Psi_{t,\eta_i}^{q(t)})(x):= P^T(\Psi_{t,\eta_i}^{q(t)})(x)[P(\Psi_{t,\eta_i}^{q(t)})P^T(\Psi_{t,\eta_i}^{q(t)})(x)]^{-1}.$$ 
\end{theo}
\begin{proof}
It is sufficient to present the proof at a point $x\in M$  with normal coordinates near the point $x$.\par
By linear algebra, since the rank of $P_c(\Psi_{t,\eta_i}^{q(t)})(x)$ is of $\frac{n(n+3)}{2}-1$ (Proposition \ref{P_c ker}), and $[0, h]^T(x)\in \mathrm{Im}P_c(\Psi_{t,\eta_i}^{q(t)})(x)$, we know for $v(x)\perp \mathrm{Ker} P(\Psi_{t,\eta_i}^{q(t)})(x)$, the linear system 
\begin{equation}\label{linear system of P_c}
    P_c(\Psi_{t,\eta_i}^{q(t)})(x)\cdot v(x)=\begin{bmatrix}
        0\\
        h
    \end{bmatrix}(x)
\end{equation}
has solution, and the solution space is of $1$ dimensional. Thus we see all the solutions are of the form 
\begin{equation*}
    v(x)=E(\Psi_{t,\eta_i}^{q(t)})(0,h)(x)+k(x)E(\Psi_{t,\eta_i}^{q(t)})(0,g)(x)
\end{equation*}
once the followings are verified: \begin{enumerate}[label=(\roman*)]
    \item $v(x)=E(\Psi_{t,\eta_i}^{q(t)})(0,h)(x)$ is a special solution of \eqref{linear system of P_c}, and
    \item $E(\Psi_{t,\eta_i}^{q(t)})(0,g)(x)$ is in $\mathrm{Ker} P(\Psi_{t,\eta_i}^{q(t)})(x)$.
\end{enumerate}\par
Having the definitions of $P(\Psi_{t,\eta_i}^{q(t)})$ and $P_c(\Psi_{t,\eta_i}^{q(t)})$ as $\frac{n(n+3)}{2}\times q(t)$ matrix as in section 2.1, we have the following expression: 
\begin{equation}
    P_c(\Psi_{t,\eta_i}^{q(t)})=P(\Psi_{t,\eta_i}^{q(t)})-\frac{1}{n}\begin{bmatrix}
      0&0\\
      0& J_n
    \end{bmatrix}P(\Psi_{t,\eta_i}^{q(t)}).
\end{equation}
Moreover, for $h=f-\frac{\tr_g f}{n}g$, $f\in\Gamma( \mathrm{Sym}^{\otimes 2}T^*M)$, and let $f_{ij}$ be the coefficient of the 2-tensor $f$, we write $(0,h)^T(x)$ with diagonal in the last $n$ rows  as in Notation \eqref{Notaion: large matrix}:

\begin{equation*}
    \begin{bmatrix}
      0\\
      h
    \end{bmatrix}(x)=\begin{bmatrix}
      0 &
      \dots  &
      0  &
      f_{1\,2}&
      \dots&
      f_{n-1\, n}&
      f_{1\,1}-\frac{1}{n}\sum\limits_{k=1}^n f_{k\,k}&
      \dots &
      f_{n\,n}-\frac{1}{n}\sum\limits_{k=1}^n f_{k\,k}
    \end{bmatrix}^T(x).
\end{equation*}
Then in such expression we obtain $\begin{bmatrix}
      0&0\\
      0& J_n
    \end{bmatrix}\begin{bmatrix}
  0\\h
\end{bmatrix}(x)=0$.
Hence for $v(x)=E(\Psi_{t,\eta_i}^{q(t)})(0,h)(x)$, 
\begin{equation*}
\begin{aligned}
    P_c(\Psi_{t,\eta_i}^{q(t)})(x)v(x)&=P_c(\Psi_{t,\eta_i}^{q(t)})(x)P^T(\Psi_{t,\eta_i}^{q(t)})[P(\Psi_{t,\eta_i}^{q(t)})P^T(\Psi_{t,\eta_i}^{q(t)})]^{-1}(x)\begin{bmatrix}
  0\\
  h
\end{bmatrix}(x)\\
&=(I_{\frac{n(n+3)}{2}}-\frac{1}{n}\begin{bmatrix}
      0&0\\
      0& J_n
    \end{bmatrix})\begin{bmatrix}
      0\\h
    \end{bmatrix}(x)=\begin{bmatrix}
      0\\h
    \end{bmatrix}(x),
\end{aligned}
\end{equation*}
this proves (i). \par
The proof is done if we show (ii). Indeed, direct computation gives
\begin{equation*}
\begin{aligned}
    P_c(\Psi_{t,\eta_i}^{q(t)})(x)E(\Psi_{t,\eta_i}^{q(t)})(0&,g)(x)=P_c(\Psi_{t,\eta_i}^{q(t)})(x)P^T(\Psi_{t,\eta_i}^{q(t)})[P(\Psi_{t,\eta_i}^{q(t)})P^T(\Psi_{t,\eta_i}^{q(t)})]^{-1}(x)\begin{bmatrix}
  0\\
  g
\end{bmatrix}(x)\\
=&(I_{\frac{n(n+3)}{2}}-\frac{1}{n}\begin{bmatrix}
      0&0\\
      0& J_n
    \end{bmatrix})\begin{bmatrix}
         0&
         \cdots&
          0&
      1&
      \cdots&
      1
    \end{bmatrix}^T= \begin{bmatrix}
      0&
      \cdots&
      0   \end{bmatrix}^T.
\end{aligned}
\end{equation*}

\end{proof}

\subsection{Estimates about $E$.} In order to apply the implicit function theorem, we need to prepare the estimates of the  norms $\rVert E(\Psi_{t,\eta_i}^{q(t)})\lVert_{C^{r,\alpha}}$ and  $\rVert E(\Psi_{t,\eta_i}^{q(t)})(0,h)\lVert_{C^{r,\alpha}} $, where the latter represents the right inverse operator with $(0,h)$ as the input.\par
The following analytic preliminaries will be used in latter computations.
\begin{lemm}
    
\begin{enumerate}[label=(\roman*)]
\item If $u\in C^{r+1}(M)$, and manifold $M$ is compact, then for a fixed $0< \alpha <1$, there is a $C$ just about $M, \alpha, r$, such that 
\begin{equation*}
    ||u||_{C^{r,\alpha}}<C||u||_{C^{r+1}}.
\end{equation*}

\item Assuming the norm on the right-hand side exists, after using the former observation we have for functions $u,v$ on compact $M$, and constant $C_r$ about $r$
\begin{equation*}
    ||uv||_{C^{r,\alpha}}<C_r||u||_{C^{r,\alpha}}||v||_{C^{r,\alpha}},
\end{equation*}
and even a finer estimate, for $0\leq r_0<r$, and constant $C_r$ only about $r$:
\begin{equation*}
    ||uv||_{C^{r,\alpha}}<C_r(||u||_{C^{r_0,\alpha}}||v||_{C^{r,\alpha}}+||v||_{C^{r_0,\alpha}}||u||_{C^{r,\alpha}}+||u||_{C^{r-1,\alpha}}||v||_{C^{r-1,\alpha}}).
\end{equation*}
\end{enumerate}
\end{lemm}

\begin{prop} \label{derivative estimate of P}
Let $\Vec{\gamma}$ be a multi-index, and $|\Vec{\gamma}|$ be the sum of its components. As $t\rightarrow0_+$, the H\"{o}lder derivatives satisfy
\begin{equation*}
   \big[D^{\Vec{\gamma}}\Psi_{t,\eta_i}^{q(t)}(x)\big]_{\alpha;M}\leq Ct^{-\frac{|\Vec{\gamma}|-1+\alpha}{2}}, \,  \rVert\Psi_{t,\eta_i}^{q(t)}(x)\lVert_{C^{r,\alpha}(M)}\leq Ct^{-\frac{r-1+\alpha}{2}}
\end{equation*}
for some constant $C>0$.
\end{prop}
\begin{proof}
Recall that $\Phi_t$ is defined in Definition \ref{defi of psi_t}. The estimate about $\Phi_t$ is due to \cite[Proposition 24]{WZ}, which is 
\begin{equation*}
    \big[ D^{\Vec{\gamma}}\Phi_t(x)\big]_{\alpha;M}\leq Ct^{-\frac{n}{4}-\frac{|\Vec{\gamma}|+\alpha}{2}};\, \lVert \Phi_t(x)\rVert_{C^{r,\alpha}}\leq Ct^{-\frac{n}{4}-\frac{r+\alpha}{2}}.
\end{equation*}
Since the normalized $\Psi_t$ is defined by $\Psi_t=\sqrt{2}(4\pi)^{n/4}t^{\frac{n+2}{4}} \Phi_t$, and the estimates on the truncated tail of $\Psi_{t,\eta_i}^{q(t)}$ in Proposition \ref{weyl law},  we will have the inequalities of H\"older derivatives.
\end{proof}\begin{prop}\label{derivative estimate of E} 
For $q(t)\geq  t^{-\frac{n}{2}-1}$, $\eta_i\in C^\infty(M,\mathbb{R})$ given as in Proposition \ref{part 1 of mt}, $E(\Psi_{t,\eta_i}^{q(t)})$ has the $C^{r,\alpha}$ estimate and the operator norm 
\begin{equation}
   \rVert E(\Psi_{t,\eta_i}^{q(t)})\lVert_{C^{r,\alpha}(M)}\leq Ct^{-\frac{r+\alpha}{2}}, \quad
  \rVert E(\Psi_{t,\eta_i}^{q(t)})\lVert_{op}\leq Ct^{-\frac{r+\alpha}{2}}
\end{equation}
for a constant $C$ that depends, in particular, on $\eta_i$,  among others.
\end{prop}
\begin{proof}
    This proposition is due to \cite[Corollary 31]{WZ} on the estimates of $ \rVert E(\Psi_{t,\eta_i}^{q(t)})\lVert_{C^{r,\alpha}}$ and $   \rVert E(\Psi_{t,\eta_i}^{q(t)})\lVert_{op}$, and Proposition \ref{weyl law} on the estimates of truncated tail.
\end{proof}

\section{G\"unther's implicit function theorem}
As discussed in Remark \ref{Nash's simplification} and Theorem \ref{construction of E_c}, our goal is to seek the solutions of $v\in C^{r,\alpha}(M, \mathbb{R}^{q(t)})$, $r\geq2$, in the following equation: 
\begin{equation}\label{iteration equation}
 E(\Psi_{t,\eta_i}^{q(t)})(0,-\frac{1}{2}h) +kE(\Psi_{t,\eta_i}^{q(t)})(0, g)  =E(\Psi_{t,\eta_i}^{q(t)})(0,-\frac{1}{2}h+k\cdot g)=v-Q_{\Psi_{t,\eta_i}^{q(t)}}(v,v).
\end{equation}
 Recall that the definition of $Q_{\Psi_{t,\eta_i}^{q(t)}}(v,v)$ is given as in \eqref{defi of Q(v,v)} with $u$ substituted by $\Psi_{t,\eta_i}^{q(t)}$.

\par
\begin{lemm}[\protect{\cite[Proposition 33]{WZ}}]\label{Lemma: Q(v,v) estimates}
 Given $\eta_i\in C^\infty(M,\mathbb{R})$, for any $v\in C^{r,\alpha}(M,\mathbb{R}^q)$, $q=q(t)\geq t^{-\frac{n}{2}-1}$, we have 
\begin{equation*}
\rVert Q_{\Psi_{t,\eta_i}^{q(t)}}(v,v)\lVert_{C^{r,\alpha}(M,\mathbb{R}^q)}\leq C(r,\alpha,M,g,\eta_i)t^{-\frac{r+\alpha}{2}}\rVert v \lVert^2_{C^{r,\alpha}(M,\mathbb{R}^q)}   . 
\end{equation*}
\end{lemm}
\begin{proof}
    This lemma primarily follows easily from \cite[Proposition]{WZ} by passing to $\Psi_{t,\eta_i}^{q(t)}$ from Proposition \ref{weyl law}.  Notice that 
    \begin{equation*}
        \rVert Q_{\Psi_{t,\eta_i}^{q(t)}}(v,v)\lVert_{C^{r,\alpha}(M,\mathbb{R}^q)}\leq C(r,\alpha,M,g)
        \rVert E(\Psi_{t,\eta_i}^{q(t)})\lVert_{C^{r,\alpha}(M)}
        \rVert v \lVert^2_{C^{r,\alpha}(M,\mathbb{R}^q)},   
    \end{equation*}
    and since the constant of the inequality controlling $\|E(\Psi_{t,\eta_i}^{q(t)})\|_{C^{r,\alpha}(M)}$ in Proposition \ref{derivative estimate of E} depends on $\eta_i$, we see the constant in the inequality of this Lemma is moreover about $\eta_i$.
\end{proof}

\begin{rema}
Following the definition of $Q_{\Psi_{t,\eta_i}^{q(t)}}(v,v)$ and the former lemma, we note that $Q_{\Psi_{t,\eta_i}^{q(t)}}$ is a bilinear operator, which also has a norm estimate. For our purpose, we only need the following for $u,v \in C^{r,\alpha}(M,\mathbb{R}^q)$:
\begin{equation}\label{quadratic estimate}
\begin{aligned}
    &\rVert Q_{\Psi_{t,\eta_i}^{q(t)}}(v,v)-Q_{\Psi_{t,\eta_i}^{q(t)}}(u,u)\lVert_{C^{r,\alpha}(M,\mathbb{R}^q)}\\
    \leq &C(r,\alpha,M,g,\eta_i)t^{-\frac{r+\alpha}{2}}(\rVert v-u \lVert_{C^{r,\alpha}(M,\mathbb{R}^q)} )(\rVert v\lVert_{C^{r,\alpha}(M,\mathbb{R}^q)}+\rVert  u \lVert_{C^{r,\alpha}(M,\mathbb{R}^q)}) .
\end{aligned}
\end{equation}
\end{rema}
The following theorem  states the unique existence of the solution of \eqref{iteration equation} under a control condition of $E(\Psi_{t,\eta_i}^{q(t)})$.
\begin{theo}\label{Banach fixed point}
  Assume $\Psi_{t,\eta_i}^{q(t)}$ defined as proceeding, particularly it is a free mapping, and the not necessarily traceless remainder $h\in C^{r,\alpha}(M,\mathrm{Sym}^{\otimes 2}T^*M)$ with $r\geq 2$ , then there exists a constant $\theta $ uniform for all $q(t)\geq t^{-\frac{n}{2}-1}$ that satisfies the property: if 
\begin{equation}\label{on theta}
    t^{-\frac{r+\alpha}{2}} \cdot \rVert E(\Psi_{t,\eta_i}^{q(t)})(0,h)\lVert_{C^{r,\alpha}}<\theta,
\end{equation}
then the following fixed point equation has a unique solution in $C^{r,\alpha}(M,\mathbb{R}^q)$:
\begin{equation*}
    E(\Psi_{t,\eta_i}^{q(t)})(0,-\frac{1}{2}h)+Q_{\Psi_{t,\eta_i}^{q(t)}}(v,v)=v.
\end{equation*}
\end{theo} 
    The proof of this theorem is essentially due to the main theorem of G\"unther \cite{G1}. For the completeness of the article, we will present it here.
    \begin{proof}
    To prove this fixed point theorem, we  first find the solution $v\in C^{2,\alpha}$ and  subsequently establish its regularity as $C^{r,\alpha}$, $r\geq 3$.
    The initial step is defining $v_0=0$ and, for $\tau=0,1,2,\cdots $, iterating  $v_\tau$ by the equation:
    \begin{equation}\label{definition of v_t}
        v_{\tau+1}:= E(\Psi_{t,\eta_i}^{q(t)})(0,-\frac{1}{2}h)+Q_{\Psi_{t,\eta_i}^{q(t)}}(v_\tau,v_\tau).
    \end{equation}
    Our objective is to show that the sequence of $\{v_\tau\}$ converges in $C^{2,\alpha}$. Using Lemma \ref{Lemma: Q(v,v) estimates}, we have
    \begin{equation*}
        \rVert v_{\tau+1}\lVert_{C^{2,\alpha}} \leq C(2,\alpha ,M,g,\eta_i) t^{-\frac{2+\alpha}{2}}\rVert v_\tau \lVert^2_{C^{2,\alpha}} +\frac{1}{2}\rVert E(\Psi_{t,\eta_i}^{q(t)})(0,h)\lVert_{C^{2,\alpha}}.
    \end{equation*}
   By imposing the $\theta$ in \eqref{on theta} to satisfy the condition
    \begin{equation*}
    C(2,\alpha ,M,g,\eta_i) t^{-\frac{2+\alpha}{2}} \cdot \rVert E(\Psi_{t,\eta_i}^{q(t)})(0,h)\lVert_{C^{2,\alpha}}<\frac{1}{2},
    \end{equation*}
    we obtain 
    \begin{equation*}
        2\rVert E(\Psi_{t,\eta_i}^{q(t)})(0,h)\lVert_{C^{2,\alpha}}\cdot \rVert v_{\tau+1}\lVert_{C^{2,\alpha}} < \rVert v_\tau \lVert^2_{C^{2,\alpha}} +\rVert E(\Psi_{t,\eta_i}^{q(t)})(0,h)\lVert^2_{C^{2,\alpha}}.
    \end{equation*}
    Therefore, by induction from $\tau=0$, we have that for all $\tau\geq 0$ 
    \begin{equation}\label{control of v}
        \rVert v_\tau\lVert_{C^{2,\alpha}} <\rVert E(\Psi_{t,\eta_i}^{q(t)})(0,h)\lVert_{C^{2,\alpha}}.
    \end{equation}\par
    Next, we need to show $\{v_\tau\}$ is a Cauchy sequence:
    \begin{equation*}
        \begin{aligned}
            \rVert v_{\tau+1}-v_\tau\lVert_{C^{2,\alpha}}&\leq C(2,\alpha ,M,g,\eta_i) t^{-\frac{2+\alpha}{2}}\rVert v_\tau-v_{\tau -1}\lVert_{C^{2,\alpha}}\cdot (\rVert v_\tau\lVert_{C^{2,\alpha}}+\rVert v_{\tau -1}\lVert_{C^{2,\alpha}})\\
            &\leq 2 C(2,\alpha ,M,g,\eta_i) t^{-\frac{2+\alpha}{2}}\rVert E(\Psi_{t,\eta_i}^{q(t)})(0,h)\lVert_{C^{2,\alpha}}\cdot\rVert v_\tau-v_{\tau -1}\lVert_{C^{2,\alpha}},
        \end{aligned}
    \end{equation*}
    where we applied the quadratic estimate \eqref{quadratic estimate} to the definition of $v_\tau$ \eqref{definition of v_t}. By enforcing a stronger condition $2 C(2,\alpha ,M,g,\eta_i) t^{-\frac{2+\alpha}{2}}\rVert E(\Psi_{t,\eta_i}^{q(t)})(0,h)\lVert_{C^{2,\alpha}}<\frac{1}{2}$, we  obtain $\rVert v_{\tau+1}-v_\tau\lVert_{C^{2,\alpha}}<\frac{1}{2}\rVert v_{l}-v_{\tau -1}\lVert_{C^{2,\alpha}}$, demonstrating that $\{v_\tau\}$ is indeed a Cauchy sequence. Hence, we identify a unique solution $v\in C ^{2,\alpha}$ as the limit of the bounded Cauchy sequence $\{v_\tau\}$.\par
    Finally, we shall extend the regularity of the solution  $v\in C^{2,\alpha}$ to  $C^{r,\alpha }$ for $r\geq 3$, which is achieved by showing that $\rVert v_\tau\lVert_{C^{r,\alpha}}$ is bounded. Similar to the $C^{2,\alpha}$ case, we have
    \begin{equation*}
        \rVert v_{\tau+1}\lVert_{C^{r,\alpha}} \leq C(r,\alpha ,M,g,\eta_i) t^{-\frac{r+\alpha}{2}}\rVert v_\tau \lVert^2_{C^{r,\alpha}} +\frac{1}{2}\rVert E(\Psi_{t,\eta_i}^{q(t)})(0,h)\lVert_{C^{r,\alpha}},
    \end{equation*}
    if we again enforce: 
    \begin{equation*}
    C(r,\alpha ,M,g,\eta_i) t^{-\frac{r+\alpha}{2}} \cdot \rVert E(\Psi_{t,\eta_i}^{q(t)})(0,h)\lVert_{C^{r,\alpha}}<\frac{1}{2},
    \end{equation*}
    then we have $ \rVert v_\tau\lVert_{C^{r,\alpha}} <\rVert E(\Psi_{t,\eta_i}^{q(t)})(0,h)\lVert_{C^{r,\alpha}}$. Notice that here $t\rightarrow0_+$, so requiring: 
    \begin{equation*}
        C(r,\alpha ,M,g,\eta_i)  \cdot \rVert E(\Psi_{t,\eta_i}^{q(t)})(0,h)\lVert_{C^{r,\alpha}}<\frac{1}{2}t^{\frac{r+\alpha}{2}}
    \end{equation*} would also imply that  $\rVert E(\Psi_{t,\eta_i}^{q(t)})(0,h)\lVert_{C^{r,\alpha}}$ is bounded. Hence, the theorem got proven.
    \end{proof}
\section{The main theorem: conformal embeddings}
 In this section,   the formerly prepared propositions and theorems are employed to prove the main theorem of this paper.\par
The main theorem can be divided into two propositions: the first claims that we could find a family of conformal immersions $C_t$ depending on a function $k_t\in C^{r,\alpha}(M)$ of $O(t^l)$, hence we denote it as $C_{t,k_t}$, and the second checks that this $C_{t,k_t}$ is one to one, hence an embedding.
\begin{prop}[Conformal immersion]\label{conformal immersion}
For any integer $r\geq 2$  and $l$ satisfying $r+\alpha <l+\frac{1}{2}$,  there exists $t_0>0$ depending on $(r,\alpha,l,g,\eta_i)$, such that for the integer $q=q(t)\geq t^{-\frac{n}{2}-1},$ $0<t<t_0$, the truncated embedding $\Psi_{t,\eta_i}^{q(t)}\in C^{r,\alpha}(M,\mathbb{R}^{q(t)})$  can be perturbed to a family of conformal immersion $C_{t,k_t}$, parametrized by  $k_t\in K:=\{k_t\in C^{r,\alpha}(M,\mathbb{R})\big| \|k_t\|_{C^{r,\alpha}}=O(t^l)\}$,  such that for any $k_t\in K$, the aforementioned perturbation yields a unique $C^{r,\alpha}(M,\mathbb{R}^{q(t)})$ conformal immersion
    \begin{equation*}
        C_{t,k_t} : M\rightarrow\mathbb{R}^{q(t)}.
    \end{equation*} Moreover, the resulting conformal map satisfies:
    \begin{gather*}
       \|C_{t,k_t}-\Psi_{t,g(t),\eta_i}\|_{C^{r,\alpha}} = O(t^{l+\frac{1-r-\alpha}{2}}),\\
            \|C_{t,k^{a}_t}-C_{t,k^{b}_t} \|_{C^{r,\alpha}}< C(r,\alpha,M,g,\eta_i) t^{-\frac{r+\alpha}{2}} \|k_t^a-k_t^b\|_{C^{r,\alpha}}, \quad \forall\, k_t^a,k_t^b \in K.
    \end{gather*}
            
\end{prop}

\begin{proof}
The proof is to apply Theorem \ref{Banach fixed point} to our case. We note  that the expression of $P(\Psi_t)P^T(\Psi_t)$ in Proposition \ref{construction of PPT}, the estimates of $E(\Psi_t)$ in Proposition \ref{derivative estimate of E}, and the Theorem \ref{Banach fixed point} all work the same way for the truncated embedding $\Psi_{t,\eta_i}^{q(t)}$, provided that the estimate for the part that is truncated off approaches  $0$ exponentially  as presented in Proposition \ref{weyl law}.\par
 The proof goes as follows. Denote $h$ as the error term $O(t^l)$ in \eqref{error term Psi_^q(t)}, that is, $h:= (\Psi_{t,\eta_i}^{q(t)})^*g_{\mathrm{can}}-\frac{\tr_g(\Psi_{t,\eta_i}^{q(t)})^*g_{\mathrm{can}}}{n}g=O(t^l)$. In order to use Theorem \ref{Banach fixed point}, it is natural to consider $K=\{k_t\in C^{r,\alpha}(M,\mathbb{R})\big| \|k_t\|_{C^{r,\alpha}}=O(t^l)\}$. Let $k_t\in K$, and use the expression of $P(\Psi_{t,\eta_i}^{q(t)})P^T(\Psi_{t,\eta_i}^{q(t)})$ from Theorem \ref{construction of PPT} and the condition  $r+\alpha < l+\frac{1}{2}$,  we have
\begin{equation}\label{importance on l}
\begin{aligned}
   & t^{-\frac{r+\alpha}{2}} \cdot \rVert E(\Psi_{t,\eta_i}^{q(t)})(0,h-2k_t\cdot g)\lVert_{C^{r,\alpha}}\\<& C t^{-\frac{r+\alpha}{2}} \Big\lVert \big[\nabla_{j_1}\nabla_{j_2} \Psi_{t,\eta_i}^{q(t)}\big]^T_{1\leq j_1\leq j_2\leq n} \cdot O(t)\cdot (h-2k_t\cdot g) \Big\rVert_{C^{r,\alpha}}\\
    <& Ct^{-\frac{r+\alpha}{2}}\cdot (t^{-\frac{r+1+\alpha}{2}})\cdot t \cdot \|h-2k_t\cdot g\|_{C^{r,\alpha}} \\
    <& C t^{-r-\alpha+\frac{1}{2}+l}\rightarrow 0 ,\qquad \text{as } t\rightarrow0_+.
\end{aligned}
\end{equation}
 Then by Theorem \ref{Banach fixed point},  for each fixed $k_t\in K$, we get the unique solution $v_{k_t}\in C^{r,\alpha}(M,\mathbb{R}^q)$ satisfying:
\begin{equation*}
   E(\Psi_{t,\eta_i}^{q(t)})(0,-\frac{1}{2}h+k_t\cdot g)+Q_{\Psi^{q(t)}_{t,\eta_i}}(v_{k_t},v_{k_t})=v_{k_t},
\end{equation*}
notice we denote it as $v_{k_t}$ for it really depends on the fixed $k_t\in C^{r,\alpha}$.  By the Theorem \ref{construction of E_c} and  the discussion in Section 3, $v_{k_t}$ satisfies
\begin{equation*}
    P_c(\Psi_{t,\eta_i}^{q(t)})\cdot (v_{k_t}-Q(v_{k_t},v_{k_t}))=\begin{bmatrix}
        0\\
        h
    \end{bmatrix}.
\end{equation*}
Thus by the discussion in Section 3, such a $v_{k_t}$ solves 
\begin{equation*}
    h=\mathrm{tr}^\perp_g(\nabla \Psi_{t,\eta_i}^{q(t)}\cdot \nabla v_{k_t})+\mathrm{tr}^\perp_g(\nabla v_{k_t}\cdot \nabla \Psi_{t,\eta_i}^{q(t)})
    +\mathrm{tr}^\perp_g(\nabla v_{k_t}\cdot \nabla v_{k_t}).
\end{equation*}
Hence  we attain the desired conformal immersions:
\begin{equation}
    C_{t,k_t}:=\Psi_{t,\eta_i}^{q(t)}+v_{k_t}.
\end{equation}
\par
 For the difference term $v_{k_t}$, as  the computation shown in \eqref{control of v}, we know $\lVert v_{k_t}\rVert_{C^{r,\alpha}}<\lVert E(\Psi_{t,\eta_i}^{q(t)})(0,h+2k_t\cdot g)\rVert_{C^{r,\alpha}}<C t^{l+\frac{1}{2}-\frac{r+\alpha}{2}}$, hence 
 \begin{equation*}
     \|C_{t,k_t}-\Psi_{t,g(t),\eta_i}\|_{C^{r,\alpha}}\leq\lVert v_{k_t}\rVert_{C^{r,\alpha}}+ \|\Psi_{t,\eta_i}^{q(t)}-\Psi_{t,g(t),\eta_i}\|<Ct^{l+\frac{1-r-\alpha}{2}}.
 \end{equation*}
Moreover, using Proposition \ref{derivative estimate of E} and \eqref{quadratic estimate}, we have 
 \begin{align}
    & \|C_{t,k_t^a}-C_{t,k_t^b}\|_{C^{r,\alpha}}=\|v_{k_t^a}-v_{k_t^b}\|_{C^{r,\alpha}}\notag\\
    \leq &\|E(\Psi_{t,\eta_i}^{q(t)})(0,k_t^a\cdot g-k_t^b\cdot g)\|_{C^{r,\alpha}}+\|Q_{\Psi^{q(t)}_{t,\eta_i}}(v_{k_t^a},v_{k_t^a})-Q_{\Psi^{q(t)}_{t,\eta_i}}(v_{k_t^b},v_{k_t^b})\|_{C^{r,\alpha}}\notag\\
    \leq&C(r,\alpha,M,g,\eta_i)t^{-\frac{r+\alpha}{2}}\left(\|k_t^a-k_t^b\|_{C^{r,\alpha}}+ \lVert v_{k_t^a}-v_{k_t^b}\rVert_{C^{r,\alpha}}\left(\lVert v_{k_t^a}\rVert_{C^{r,\alpha}}+\lVert v_{k_t^b}\rVert_{C^{r,\alpha}}\right)\right)\notag\\
    <&C(r,\alpha,M,g,\eta_i)t^{-\frac{r+\alpha}{2}}\left(\|k_t^a-k_t^b\|_{C^{r,\alpha}}+  t^{l+\frac{1}{2}-\frac{r+\alpha}{2}}\lVert v_{k_t^a}-v_{k_t^b}\rVert_{C^{r,\alpha}} \right),\label{equ-last inequality of v_1-v_2}
 \end{align}
 where we used that  $\lVert v_{k_t}\rVert_{C^{r,\alpha}}<C t^{l+\frac{1}{2}-\frac{r+\alpha}{2}}$ in the last inequality. Since $l+\frac{1}{2}>r+\alpha$, there exists a $t_0$ such that $\forall t\in(0,t_0)$,
 \begin{equation*}
   C(r,\alpha,M,g,\eta_i) t^{l+\frac{1}{2}-r-\alpha}<\frac{1}{2},
 \end{equation*}
     where the constant $C(r,\alpha,M,g,\eta_i)$ is as in \eqref{equ-last inequality of v_1-v_2}. Thus, we obtain
     \begin{equation*}
         \frac{1}{2}\|v_{k_t^a}-v_{k_t^b}\|_{C^{r,\alpha}}=\frac{1}{2}\|C_{t,k_t^a}-C_{t,k_t^b}\|_{C^{r,\alpha}}<C(r,\alpha,M,g,\eta_i)t^{-\frac{r+\alpha}{2}}\|k_t^a-k_t^b\|_{C^{r,\alpha}}.
     \end{equation*}
     This concludes the proof. 
\end{proof}
\begin{rema}
    Regarding the inequality $ \|C_{t,k^{a}_t}-C_{t,k^{b}_t} \|_{C^{r,\alpha}}< C(r,\alpha,M,g,\eta_i) t^{-\frac{r+\alpha}{2}} \|k_t^a-k_t^b\|_{C^{r,\alpha}},$  $  \forall\, k_t^a,k_t^b \in K,$ it is natural to consider the $C^{r,\alpha}$ norm with $r\geq 2$, $0<\alpha<1$,  instead of the lower order norms such as $C^{0,\alpha}$ or $C^{1,\alpha}$. Indeed, 
recalling the definition of $Q_{u}(v,v)$ in \eqref{defi of Q(v,v)} and using the notation of $E(u)$, we have 
\begin{align*}
    Q_u(v,v)&=E(u)\left(-(\Delta-1)^{-1}\left(\Delta v\cdot \nabla v\right),M_{ij}(v)\right). 
 \end{align*}
When estimating $\lVert Q_u(v,v) \rVert_{C^{r,\alpha}}$, we observe that classical Schauder estimates provide that  $\lVert(\Delta-1)^{-1}(\Delta v\cdot \nabla v)\rVert_{C^{r,\alpha}}<C\lVert v \rVert^2_{C^{r,\alpha}}$ only for $r\geq 2$, $0<\alpha<1$; a similar requirement holds for the term $\lVert M_{ij}(v)\rVert_{C^{r,\alpha}}$.
\end{rema}
\begin{prop}[Injectivity]\label{Inject}
Let $(M,g)$ be a compact Riemannian manifold with smooth metric $g$. Then, there exists a positive constant $\delta_0$ such that for $0<t\leq \delta_0$ and $q(t)\geq t^{-\frac{n}{2}-1}$, the truncated heat kernel mapping $\Psi_{t,\eta_i}^{q(t)}:M\rightarrow \mathbb{R}^{q(t)}$ possesses the property of point distinguishability. In other words, for any $x\neq y$ in $M$, one has $\Psi_{t,\eta_i}^{q(t)}(x)\neq \Psi_{t,\eta_i}^{q(t)}(y)$. The property of point distinguishability also holds for the perturbed almost conformal immersion  $\Psi_{t,g(t),\eta_i}$ (as defined in Definition \ref{defi-canonical almost conformal embedding}), and so is the conformal mapping $C_{t,k_t}$ for any $k_t\in C^{r,\alpha}(M)$ of $O(t^l)$.
\end{prop}
\begin{proof}
It can be easily obtained by the same argument in \cite[Proposition 36]{WZ}, since the almost conformal mapping $\Psi_{t,g(t),\eta_i}$ is also the heat kernel of some metric $g_{t}$.
\end{proof}
\begin{rema}
From Proposition \ref{part 1 of mt} and \ref{conformal immersion}, we see that the resulting conformal embeddings $C_{t,k_t}$  satisfy
\begin{equation}
    \begin{aligned}
    (C_{t,k_t})^*g_\mathrm{can}=&g+t(\frac{1}{n}\tr_g A_1(g)+\eta_1) g+t^2(\frac{1}{n}\tr_g (A_2(g)+A_{1,1}(h_1))+\eta_2) g\\
    &+\cdots+t^{l-1}(\frac{1}{n}\tr_g\sum_{i+j} A_{i,j}(h_1,\cdots,h_j)+\eta_{l-1})g+k_t \cdot g  , 
    \end{aligned}
\end{equation}
where $A_{i,j}$'s are as defined in Theorem \ref{thm-mainthm of BBG} and Proposition \ref{part 1 of mt}, and $k_t\in C^{r,\alpha}(M,\mathbb{R})$ is of $O(t^l)$. \par
This can be compared with the canonical isometric embeddings $I_t$ constructed by  by Wang-Zhu  (cf. \cite[Theorem 1]{WZ}), for which $I_t^*g_{\rm can}=g$.
\end{rema}

\bibliographystyle{abbrv}

\Addresses

\end{document}